\nonstopmode
\documentclass[12pt]{amsart}
 
\usepackage{amsmath}
\usepackage{amscd} 
\usepackage{amssymb}
\usepackage{mathtools}
\usepackage[UglyObsolete,tight,heads=LaTeX]{diagrams}

\numberwithin{equation}{section}

\theoremstyle{plain} %% This is the default, anyway

\newtheorem{theorem}[equation]{Theorem}
\newtheorem{lemma}[equation]{Lemma}
\newtheorem{proposition}[equation]{Proposition}
\newtheorem{corollary}[equation]{Corollary}

\theoremstyle{definition}
\newtheorem{defn}[equation]{Definition}
\newtheorem{definition}[equation]{Definition}

\newtheorem{remark}[equation]{Remark}
\newtheorem*{remark*}{Remark}
\newtheorem*{remarks*}{Remarks}

\newtheorem{example}[equation]{Example}

\theoremstyle{plain}
\newtheorem*{maintheorem}
       {Theorem~\ref{theorem: main theorem}}
\newcommand{\maintheoremtext}
{
Fix a prime $p$. Let $\coeff$ be a Mackey functor for
$\Sigma_{n}$ that takes values in $\integers_{(p)}$-modules. Assume
the following. 
\begin{enumerate}
\item The Mackey functor $\coeff$ is projective relative to the
  collection of $p$-subgroups of $\Sigma_n$.
\item \label{condition: centralizer} For every elementary abelian
  $p$-subgroup $D\subset \Sigma_n$ that acts freely and
  non-transitively on $\{1, \ldots, n\}$, the kernel of the
  homomorphism $\Cen_{\Sigma_n}(D)\rightarrow
  {\pi_{0}\Cen_{\GL_n\reals}(D)}$ acts trivially on
  $\coeff(\Sigma_n/D)$.
\item \label{condition: involution} If $p$ is odd and $D$ is as above, then every odd involution in $\Cen_{\Sigma_n}(D)$ acts on $\coeff(\Sigma_n/D)$ by multiplication by $-1$.
\end{enumerate}
Then if $n$ is not a power of $p$, the groups
$\BredonHomRed{*}{\Sigma_{n}}{\Pcal_n^{\diamond}}{\coeff}$
and  
$\BredonCohRed{*}{\Sigma_{n}}{\Pcal_n^{\diamond}}{\coeff}$
vanish.
If $n=p^k$, then the map%~\eqref{equation:fundamental-map}
\[
{\Sigma_n}_+ \wedge_{\Aff_k}(E\GL_k{}_+\wedge B_k^\diamond 
          ) \longrightarrow \Pcal_n^\diamond
\]
induces an isomorphism on $\BredonHomRed{*}{\Sigma_n}{-}{\coeff}$ 
and on $\BredonCohRed{*}{\Sigma_n}{-}{\coeff}$. 
}

\newtheorem*{onlyonedimensiontheorem}
{Corollary~\ref{corollary:only-one-dimension}}
\newcommand{\onlyonedimensiontext} 
{ In the setting of
  Theorem~\ref{theorem: main theorem}, suppose that $n=p^k$.  Then
there are isomorphisms
\[
\BredonHomRed{j}{\Sigma_{n}}{\Pcal_n^{\diamond}}{\coeff}
      \cong \left\{ \begin{array}{cl}
      0 & j\ne k-1 \\
     \coeff(\Sigma_n/\Delta_k) \otimes_R \Stein_k  & j=k-1 \end{array}\right.
\]      
Moreover, there are isomorphisms for all $j$ between Bredon homology and cohomology groups: 
$\BredonHomRed{j}{\Sigma_{n}}{\Pcal_n^{\diamond}}{\coeff}\cong \BredonCohRed{j}{\Sigma_{n}}{\Pcal_n^{\diamond}}{\coeff} $ for all $j\ge 0$.
}

\newcommand{\pruningcriteriontext} { 
  Let $\mu$ be a homological (or cohomological) coefficient system
  for $G$. Let $\Ccal$ be a collection of $p$-subgroups of $G$ that is
  closed under passage to $p$-supergroups.  Suppose that for each
  $K\in\Iso(X)$ and each $p$-subgroup $D\subset K$ with
  $D\notin\Ccal$, the collection of non-identity $p$-subgroups of
  $W_K(D)$ is homologically (or cohomologically) $\mu(G/D)$-ample.
  Then the map $X_{\Ccal}\to X_{\pgroupsnone}$ induces an isomorphism on
  Bredon homology (or cohomology) with coefficients in~$\mu$.  }

\newtheorem*{takeawayone}{Lemma~\ref{lem:take-away-one}}
\newcommand{\takeawayonetext} { 
  Let $\mu$ be a homological (or cohomological) coefficient system for
  $G$. Let $D$ be a subgroup of $G$, and let $\Dcal$ denote the set of
  conjugates of $D$ in $G$.  Let $\Ccal$ be a collection of
  $p$-subgroups of $G$ for which $D$ is a minimal element, and such
  that $\Ccal$ contains all $p$-supergroups of $D$.  Let
  $W=W_G(D)$. Then $(*)_{\Ccal\setminus \Dcal}\to(*)_{\Ccal}$ gives
  isomorphisms on Bredon homology (resp. cohomology) with coefficients
  in $\mu$ if and only if $\pgroupsntof{W}$ is homologically
  (resp. cohomologically) $\mu(G/D)$-ample.  }

\newcommand{\algebraicpruningtext}
{ Assume that $W$ is a finite group and that 
$M$ is a $\integers_{(p)}[W]$-module, and suppose that 
  there exists an element of order $p$ in $W$ that acts trivially on~$M$.
  Then $\pgroupsntof{W}$ is homologically and cohomologically $M$-ample.
}

\newtheorem*{badsubgroups}
       {Proposition~\ref{proposition: bad subgroups}}
\newcommand{\badsubgroupstext}
{
  Let $H$ be a problematic $p$-subgroup of~$\Sigma_{n}$. 
  Then $H$ is an elementary abelian $p$-group, and the action of $H$ on
  $\n$ is free.  }

\newtheorem*{allpgroups}
       {Proposition~\ref{proposition: family}}
\newcommand{\allpgroupstext}
{
  Let $\family$ be a family of subgroups of~$G$, let $X$ be a $G$-space, 
  and let $\coeff$ be a Mackey functor for $G$ that is projective relative 
  to~$\family$. Then $X_{\family}\rightarrow X$ induces an isomorphism 
  on $\BredonHom{*}{G}{-}{M}$ and~$\BredonCoh{*}{G}{-}{M}$.
}

\newtheorem*{VariableNoNum}{{\VariableText}}
\newtheorem{Variable}[equation]{{\VariableText}}

\theoremstyle{definition}
\newtheorem*{VariableNoNumBold}{{\VariableText}}
\newtheorem{VariableBold}[equation]{{\VariableText}}

\theoremstyle{definition}

\newlength{\asidelength}

\def\Changed/{\ifvmode\else\vadjust{%
\vbox to 0pt{\vskip -\baselineskip%
\hbox to 0pt{\hss\vrule height 0pt depth 1.2\baselineskip\hskip 1em}\vss}}\fi}
\def\CHanged{\ifvmode\else\vadjust{%
\vbox to 0pt{\vskip -\baselineskip%
\hbox to 0pt{\hss\vrule height 0pt depth 1.2\baselineskip\hskip 1em}\vss}}\fi}

\def\Math#1{\def\MathString{#1}\futurelet\MathDelim\MathChoose}
\def\MathChoose{\ifmmode\let\MathDo\MathString%
              \else\let\MathDo\MathSkip\fi%
              \MathDo}
\def\MathSkip{\ifx\MathDelim/\def\MathDo{$\MathString$\EatOne}%
              \else\def\MathDo{$\MathString$}\fi%
              \MathDo}

\def\Text#1{\def\TextString{#1}\futurelet\TextDelim\TextSkip}
\def\TextSkip{\ifx\TextDelim/\def\TextDo{\TextString\EatOne}%
              \else\let\TextDo\TextString\fi%
              \TextDo}
\def\EatOne#1{}

\def\SkipToEndScan#1\EndScan{}
\def\Scan#1#2#3{\ifx#1#2#3\expandafter\SkipToEndScan\fi\Scan#1}
\def\Upper#1{%
\Scan#1aAbBcCdDeEfFgGhHiIjJkKlLmMnNoOpPqQrRsStTuUvVwWxXyYzZ#1#1\EndScan}
\def\Phrase#1 #2/#3/#4=#5 #6/#7/#8.{%
\expandafter\edef\csname#2#3\endcsname{\noexpand\Text{#6#7}}
\expandafter\edef\csname\Upper#2#3\endcsname{\noexpand\Text{\Upper#6#7}}
\expandafter\edef\csname#1#2#3\endcsname{\noexpand\Text{#5 #6#7}}
\expandafter\edef\csname\Upper#1#2#3\endcsname{\noexpand\Text{\Upper#5 #6#7}}
\expandafter\edef\csname#2#4\endcsname{\noexpand\Text{#6#8}}
\expandafter\edef\csname\Upper#2#4\endcsname{\noexpand\Text{\Upper#6#8}}
}
\newcommand{\orb}[1]{{#1\kern-2pt}{\scriptscriptstyle\rm{-orb}}}
\newcommand{\orbit}[2]{#1/#2}

\newcommand{\cclass}{\cbold}
\newcommand{\orbits}[2]{#1/#2} 
\newcommand{\orbitsprime}{\mu} 

\newcommand{\GL}{\operatorname{GL}}

\newcommand{\Hom}{\operatorname{Hom}}

\newcommand{\hhh}{\operatorname{h}\!}
\newcommand{\thhh}{\tilde{\operatorname{h}}}
\def\HomotopyOrbit#1on#2/{\ensuremath{#2_{\hhh#1}}}
\def\RedHomotopyOrbit#1on#2/{\ensuremath{#2_{\thhh#1}}}

\newcommand{\coeff}{M}
\newcommand{\BredonHom}[4]{H_{#1}^{#2}\left(#3; #4\right)}
\newcommand{\BredonCoh}[4]{H^{#1}_{#2}\left(#3; #4\right)}
\newcommand{\BredonHomRed}[4]{\Hwiggle_{#1}^{#2}\left(#3; #4\right)}
\newcommand{\BredonCohRed}[4]{\Hwiggle^{#1}_{#2}\left(#3; #4\right)}

\DeclareMathOperator{\Aff}{Aff}
\DeclareMathOperator{\Aut}{Aut}

\newcommand{\Cen}{C}
\DeclareMathOperator{\cl}{cl}

\DeclareMathOperator{\Diag}{Diag}

\DeclareMathOperator{\im}{im}
\DeclareMathOperator{\Iso}{Iso}
\DeclareMathOperator{\map}{map}

\DeclareMathOperator{\Stein}{St}

\DeclareMathOperator{\Tor}{Tor}

\def\doCal#1{%
\ifx#1\doAllCalEnd\def\doAllCal{\relax}\else%
 \expandafter\edef\csname#1cal\endcsname{{\noexpand\mathcal #1}}\fi}
\def\doAllCal#1{\doCal#1\doAllCal}
\doAllCal abcdefghijklmnopqrstuvwxyzABCDEFGHIJHLMNOPQRSTUVWXYZ\doAllCalEnd

\def\doBar#1{%
\ifx#1\doAllBarEnd\def\doAllBar{\relax}\else%
 \expandafter\edef\csname#1bar\endcsname{{\noexpand\overline{#1}}}\fi}
\def\doAllBar#1{\doBar#1\doAllBar}
\doAllBar ABCDEFGHIJHLMNOPQRSTUVWXYZ\doAllBarEnd
\newcommand{\taubar}{{\overline{\tau}}}
\newcommand{\lambdabar}{{\overline{\lambda}}}
\newcommand{\Lambdabar}{{\overline{\Lambda}}}

\def\doWiggle#1{%
\ifx#1\doAllWiggleEnd\def\doAllWiggle{\relax}\else%
 \expandafter\edef\csname#1wiggle\endcsname{{\noexpand\tilde{#1}}}\fi}
\def\doAllWiggle#1{\doWiggle#1\doAllWiggle}
\doAllWiggle ABCDEFGHIJHLMNOPQRSTUVWXYZabcdefghijklmnopqrstuvwxyz\doAllWiggleEnd

\newcommand{\epi}{\twoheadrightarrow}

\newcommand{\class}[1]{\cl\left(#1\right)}
\newcommand{\definedas}{\mathrel{\vcenter{\baselineskip0.5ex \lineskiplimit0pt
                     \hbox{\scriptsize.}\hbox{\scriptsize.}}}%
                     =}
\newcommand{\defining}[1]{{\emph{#1}}}

\newcommand{\ho}{\text{h}}
\newcommand{\hobased}{{\tilde{\text{h}}}}

\newcommand{\colim}{\operatorname{colim}\,}

\newcommand{\cbold}{ {\mathbf{c}} }

\newcommand{\G}{ {\bf{G}} }
\newcommand{\n}{ {\mathbf{n}} }

\renewcommand{\r}{{\mathbf{r}}}
\newcommand{\s}{{\mathbf{s}}}

\newcommand{\isogroupof}[1]{K_{#1}}

\newcommand{\family}{\Fcal}
\newcommand{\pgroupsntof}[1]{\Scal_{p}\kern-1pt\left(#1\right)}
\newcommand{\pgroupsof}[1]{{\overline{\Scal}_{p}}\kern-1pt\left(#1\right)}
\newcommand{\pgroupsntnone}{\Scal_{p}}
\newcommand{\pgroupsnone}{{\overline{\Scal}_{p}}}

\newcommand{\restrict}{|}

\newcommand{\reals}{{\mathbb{R}}}
\newcommand{\integers}{{\mathbb{Z}}}

\newcommand{\field}{{\mathbb{F}}}

\newcommand{\downto}{\downarrow}

\Phrase a superfluous//=a prunable//.
\newcommand{\myker}{C_0}
\newcommand{\mykerbar}{\Cbar_0}
\newcommand{\mykernel}{J}

\begin{document}
% Topmatter
%$Modified: Thu Nov 22 14:39:23 2012 by wgd $
%%% In the title, use a double backslash "\\" to show a linebreak:
%%% Use one of the following two forms:
%%% \title{Text of the title}
%%% or
%%% \title[Short form for the running head]{Text of the title}
\title[Bredon homology of partition complexes]
{Bredon homology of partition complexes}

\author{G. Z. Arone}
\address{Kerchof Hall, U. of Virginia, P.O. Box 400137,
         Charlottesville VA 22904 USA}
\email{zga2m@virginia.edu}
%\urladdr{http://www.math.virginia.edu/~zga2m}
%%%%%
\author{W. G. Dwyer}
\address{Department of Mathematics, University of Notre Dame,
             Notre Dame, IN 46556 USA}
\email{dwyer.1@nd.edu}
%\urladdr{http://www.nd.edu/~wgd}
%%%%%
\author{K. Lesh}
\address{Department of Mathematics, Union College, Schenectady, NY 12309 USA}
\email{leshk@union.edu}
%\urladdr{http://www.math.union.edu/~leshk}

\thanks{The authors were partially supported by NSF grants
DMS-0968221, DMS-0967061, and DMS-0968251.}

%%% To have the current date inserted, use \date{\today}:
\date{\today}

\maketitle

%%% To include a table of contents, uncomment the next line:
% \tableofcontents

\begin{abstract} 
  We prove that the Bredon homology or cohomology of the partition
  complex with fairly general coefficients is either trivial or
  computable in terms of constructions with the Steinberg module. The
  argument involves a theory of Bredon homology and cohomology approximation.
\end{abstract}

\section{Introduction}

In this paper, we study the poset obtained by ordering the partitions
of the set $\n\definedas\{1,\ldots, n\}$ by coarsening.  The partition
of $\n$ into one-element sets is called the \defining{discrete}, or
\defining{trivial} partition.  The partition consisting of the set
$\n$ itself is called \defining{indiscrete}, and partitions of $\n$
that are not indiscrete are called \defining{proper}.  With this
terminology, let $\Pcal_{n}$ denote the nerve of the poset of proper
nontrivial partitions of $\n$, and let $\Pcal_{n}^\diamond$ denote its
unreduced suspension. This space, with its natural action of the
symmetric group $\Sigma_n$, arises in various contexts, and in
particular it plays a role in the calculus of functors. We study the
Bredon homology and cohomology of $\Pcal_n^\diamond$ as a
$\Sigma_{n}$-space.

For the moment, we focus on homology. Let $G$ be a finite group and
$X$ a $G$-space, or a simplicial $G$-set. We denote the Borel
construction on $X$ by $X_{\ho G}\definedas(EG\times X)/G$. If $X$ has
a pointed $G$-action, i.e., $X$ has a basepoint that is fixed by the
$G$-action, we denote the reduced Borel construction by 
$X_{\hobased G}\definedas \left(EG_{+}\wedge X\right)/G$.  One type of
equivariant homology for $X$ is the ordinary twisted homology of the
Borel construction $X_{\ho G}$ with coefficients in a $G$-module~$M$
or, if $X$ has a pointed action, the ordinary twisted homology of
$X_{\hobased G}$. The \defining{Bredon homology} of $X$ is a finer
invariant, which takes coefficients in an additive functor~$\gamma$
from finite $G$-sets to abelian groups.  Our goal in this paper, in
rough terms, is to sharpen the results of~\cite{A-D} about the mod~$p$
homology of the Borel construction on $\Pcal_n^\diamond$ by proving
similar results about the Bredon homology of~$\Pcal_{n}^{\diamond}$
with somewhat general coefficients.

To carry out this program, we require a rather detailed analysis of the
fixed point spaces of various $p$-subgroups of $\Sigma_n$ acting 
on~$\Pcal_{n}$. In particular, we completely classify the
$p$-subgroups whose fixed point spaces on $\Pcal_{n}$ are not
contractible (Proposition~\ref{proposition: bad subgroups}). We also
need to study equivariant approximations that induce an isomorphism on
Bredon homology. Here we build on earlier work of several people, in
particular Webb and the second author of the present paper
\cite{WebbSplit, Webb, Homology, Sharp, A-D}.

We introduce some further terminology and then describe the results about
$\Pcal_{n}$ in more detail. A \defining{Mackey functor} $M$ for $G$ is
a pair of additive functors $(\gamma,\gamma^\natural)$ from finite
$G$-sets to abelian groups, where $\gamma$ is covariant and
$\gamma^\natural$ is contravariant, and $\gamma$ and
$\gamma^{\natural}$ take common values and satisfy certain conditions
(Definition~\ref{definition: Mackey functor}).  Hence $\gamma$ can
serve as a coefficient system for Bredon homology of $G$-spaces, and
$\gamma^\natural$ as a coefficient system for Bredon cohomology.  We
denote the resulting homology and cohomology groups of a $G$-space $X$
by $\BredonHom{*}{G}{X}{\coeff}$ and $\BredonCoh{*}{G}{X}{\coeff}$.
There are also reduced versions of Bredon homology and cohomology,
defined for spaces with pointed $G$-actions. They are denoted by
$\BredonHomRed{*}{G}{X}{\coeff}$ and $\BredonCohRed{*}{G}{X}{\coeff}$,
respectively (Section~\ref{section: background on Bredon homology}).

The notion of a Mackey functor that is projective relative to $p$-subgroups is
important for our main theorem, and it is reviewed in
Section~\ref{section: mackey functors}. In brief, it amounts to the
following. Let $M$ be a Mackey functor for $G$ and let $P\subset G$ be
a $p$-Sylow subgroup. There is a natural transformation of Mackey
functors $M(G/P \times -)\longrightarrow M(-)$ induced by projection.
We say $M$ is \defining{projective relative to $p$-subgroups} if this
transformation is a split epimorphism.

We mentioned earlier that we use an approximation of $\Pcal_{n}$ 
to compute Bredon homology and cohomology, which we refer
to jointly as Bredon (co)homology, for brevity. 
In~\cite{A-D}, the authors approximate the $\Sigma_{n}$-space
$\Pcal_{n}$ when $n$ is a prime power by inducing up from a
Tits building. We adapt that procedure for this work, as described
in the next two paragraphs.  Let
$p$ be a prime and suppose that $n=p^{k}$ for some positive integer
$k$. Let $B_k$ be the nerve of the poset of proper, nontrivial
subgroups of the group 
$\Delta_k\definedas \left(\integers/p\integers\right)^k$, ordered by
inclusion.  The complex $B_k$ is the Tits building for
$\GL_k\definedas\GL_k\left(\field_p\right)$.  
%, with the south pole as basepoint.
Since $n=p^k$, we can identify the underlying set of $\Delta_k$,
which has $p^k$ elements, with
$\n=\{1,\ldots,p^k\}$. We can then construct a pointed map 
$B_k\longrightarrow \Pcal_n$ by assigning to a subgroup $V\subseteq\Delta_k$,
the partition of $\n$ given by the cosets of $V$ in $\Delta_k$.
Writing $B_k^\diamond$ for the unreduced suspension of $B_{k}$, we 
obtain a map $B_{k}^\diamond\rightarrow\Pcal_{n}^{\diamond}$.

The action of $\Delta_k$ on its underlying set by left translation, and
the identification $\Delta_k\leftrightarrow\n$, allow us to identify
$\Delta_k$ as a subgroup of $\Sigma_n$.  The normalizer of
$\Delta_k$ in $\Sigma_n$ is isomorphic to the affine group
$\Aff_k \cong \Delta_{k}\rtimes \GL_{k}$, which then acts on $B_{k}$
(with $\Delta_{k}$ acting trivially). The $\Aff_{k}$-equivariant map 
$B_k^\diamond\to \Pcal_n^\diamond$ extends to a $\Sigma_n$-equivariant map
$
{\Sigma_n}_+ \wedge_{\Aff_k}(E\GL_k{}_+\wedge B_k^\diamond 
          ) \longrightarrow \Pcal_n^\diamond
$,
which turns out to be a good enough approximation to  
$\Pcal_n^\diamond$ to compute Bredon homology, as stated in
Theorem~\ref{theorem: main theorem}, below.

The following is our main theorem. If $H$ is a subgroup of $G$, then 
$\Cen_G(H)$ denotes the centralizer of $H$ in $G$. Note that $\Cen_G(H)$ 
acts on $G/H$ by $G$-equivariant maps, so if $M$ is a Mackey
functor for~$G$, there is an action of $\Cen_G(H)$ on $M(G/H)$.

\begin{theorem} \label{theorem: main theorem}
\maintheoremtext
\end{theorem}

The proof is found in Section~\ref{section:approximation-results}.
We will explain the assumptions and why they are needed near the end
of the introduction, where we outline the proof of the main theorem.

In the case $n=p^k$, Theorem~\ref{theorem: main theorem} leads to an
algebraic formula for the Bredon homology and cohomology of
$\Pcal_n^\diamond$. Let $\Stein_k$ denote
$\Hwiggle_{k-1}\left(B_k^{\diamond};\integers\right)$, 
which is the Steinberg module for
$\GL_k\left(\field_p\right)$. The group $\GL_k\left(\field_p\right)$
acts on $\coeff\left(\Sigma_n/\Delta_k\right)$.
Let $R$ denote the ring
$\integers\left[\GL_k\left(\field_p\right)\right]$. The following is a
consequence of Theorem~\ref{theorem: main theorem}. 

\begin{corollary} \label{corollary:only-one-dimension}
\onlyonedimensiontext
\end{corollary}

\begin{example} \label{example: relevant-Mackey-functor} 
  To see an example of a Mackey functor for which
  Theorem~\ref{theorem: main theorem} applies, recall that $\Sigma_n$
  acts on the one-point compactification $S^n$ of $\reals^n$ by
  permuting coordinates, and hence on the $j$-fold smash product
  $\left(S^{n}\right)^{\wedge j}=S^{nj}$.  Fix a prime $p$ and an
  integer $j$, where $j$ is required to be odd if $p\neq 2$. Let
  $E_*$ be a non-equivariant generalized homology theory that takes
  values in $\integers_{(p)}$-modules.  There is a graded Mackey functor
  $\coeff_E$ for $\Sigma_n$ that assigns to a finite $\Sigma_n$-set
  $T$ the graded abelian group
\[
 \coeff_E(T)_*=\Ewiggle_*\left(\Sigma^\infty T_+\wedge S^{nj}\right)_{\hobased
         \Sigma_n}\,.
\]
The graded constituents of $\coeff_E$ satisfy
the hypotheses of Theorem~\ref{theorem: main theorem}. This is discussed in 
detail in Section~\ref{section: our coefficients}, and in fact a more general statement is proved there.
\end{example}

Example~\ref{example: relevant-Mackey-functor} ties
Theorem~\ref{theorem: main theorem} and 
Corollary~\ref{corollary:only-one-dimension} to previous work. If $X$
is a pointed $\Sigma_n$-space, then filtering $X$ by its skeleta 
gives the ``isotropy spectral sequences''
\[
\begin{aligned}
   \BredonHomRed{a}{\Sigma_{n}}{X}{\left(\coeff_E\right)_{b}}
  &\Rightarrow \Ewiggle_{b+a}\left(X \wedge \Sigma^{\infty}S^{nj}\right)_{\hobased\Sigma_n}\\
   \BredonCohRed{a}{\Sigma_{n}}{X}{\left(\coeff_E\right)_{b}}
 &\Rightarrow  
   \Ewiggle_{b-a}\map_{*}\left(X, \Sigma^{\infty} S^{nj}\right)_{\hobased\Sigma_n}
\end{aligned}
\]
where the second is guaranteed to converge only if $X$ has a finite number
of $\Sigma_n$-cells. These spectral sequences can be used to obtain
the main theorem of~\cite{A-D} from 
Theorem~\ref{theorem: main theorem}. In effect, \cite{A-D} calculates
the abutments of these spectral sequences for $X=\Pcal_n^\diamond$
and $E=H\integers/p$, while Theorem~\ref{theorem: main theorem} and
Corollary~\ref{corollary:only-one-dimension} calculate the $E^2$-pages
in a form that implies a collapse result.

In fact, for $E_*=H_*(-;\field_p)$, the groups
$\BredonCohRed{*}{\Sigma_{n}}{\Pcal_n^\diamond}{\coeff_E}$
were first calculated in~\cite{A-M} by brute force, using detailed
knowledge of the homology of symmetric groups.  This paper gives a
new, more conceptual approach to those calculations.

\subsubsection*{Intended applications}

We are particularly interested in the graded Mackey functors
$\coeff_E$ as in Example~\ref{example: relevant-Mackey-functor} when
%\Note{This is p-local in Section 11. Which is it?}
$E$ is the $p$-local sphere spectrum. As discussed 
in Section~\ref{section: our coefficients}, this application of
Theorem~\ref{theorem: main theorem} leads to new proofs of some
theorems of Behrens and Kuhn on the relationship
between the Goodwillie tower of the identity and the symmetric power
filtration of $H\integers$, as well as another approach to Kuhn's
proof of the Whitehead Conjecture.  We will pursue this in another
paper.

Theorem \ref{theorem: main theorem} can also be applied to the Mackey functor
\[
\coeff(T)_*=\pi_*L_K\left(\left(E\wedge \Sigma^\infty T_+\wedge S^{nj}\right)_{\hobased
         \Sigma_n}\right)\,.
\]
Here $E$ is one of Morava's $E$-theories and $L_K$ denotes
localization with respect to the corresponding Morava $K$-theory. 
In this case Theorem~\ref{theorem: main theorem}
seems to offer an alternative approach to some recent
calculations of Behrens and Rezk, and again we hope to discuss this elsewhere.

\subsubsection*{Connection with other work} 
There is a connection between this paper and the work of
Grodal~\cite{Grodal}. Grodal's paper concerns the Bredon cohomology of
spaces of the form $\vert\Ccal\vert$, where $\Ccal$ is a poset of
$p$-subgroups of a group $G$. Our space $\Pcal_n$ is of a similar
nature: it is $\Sigma_n$-equivariantly equivalent to the nerve of the
poset of nontrivial non-transitive subgroups (not just $p$-subgroups)
of $\Sigma_n$.  The (generalized) Steinberg module also plays a
central role in~\cite{Grodal}. 
% It would be interesting to understand
% better how our calculation fits with Grodal's general framework.
\medskip

We devote the remainder of this introduction to outlining the
proof of Theorem~\ref{theorem: main theorem}.  Let~$G$ be a finite
group and $\Ccal$ a collection of subgroups of~$G$ (i.e., a set of
subgroups closed under conjugation).  Given a $G$-space $X$, one can
associate with it a $G$-space $X_{\Ccal}$, together with a natural map
$X_{\Ccal}\rightarrow X$, called the \defining{$\Ccal$-approximation to $X$}.
The approximation is characterized up to homotopy by
the fact that $X_{\Ccal}$ has isotropy only in $\Ccal$ and
$X_{\Ccal}\rightarrow X$ induces an equivalence on $K$-fixed points
for $K\in\Ccal$. 
The construction of ${\Ccal}$-approximations is reviewed in
Section~\ref{section: background on approximations}. 
A typical example of interest is the collection 
of nontrivial $p$-subgroups of a finite group~$G$, which we denote
$\pgroupsntof{G}$, following Quillen~\cite{Quillen}.  

The relevance of $\Ccal$-approximation is that 
$X_{\Ccal}\rightarrow X$ may induce an isomorphism on Bredon homology
without being an equivalence of $G$-spaces. To state a first result
along these lines, recall that a \defining{family} of subgroups of $G$
is a collection that is closed under taking subgroups as well as under
conjugation.  The following is a variant of the main result of
Webb~\cite{WebbSplit}.

\begin{allpgroups}
\allpgroupstext
\end{allpgroups}

If we add the trivial subgroup to the collection $\pgroupsntof{G}$, we
obtain the family of all $p$-subgroups of~$G$, which we denote
$\pgroupsof{G}$. (We suppress the group and write $\pgroupsntnone$ or
$\pgroupsnone$ if the group is clear from context.)  We will apply
Proposition~\ref{proposition: family} with $X=\Pcal_n^\diamond$ and
$\family=\pgroupsof{\Sigma_{n}}$, the family of all $p$-subgroups
of~$\Sigma_{n}$. In this way we obtain a starting approximation
$\left(\Pcal_n^\diamond\right)_{\pgroupsof{\Sigma_{n}}} \to
\Pcal_n^\diamond$ that induces an isomorphism on Bredon homology and
cohomology.

To analyze the approximation
$\left(\Pcal_n^\diamond\right)_{\pgroupsof{\Sigma_{n}}} \to
\Pcal_n^\diamond$, we need to reduce the size of the approximating
collection.  Our key criterion for eliminating (or ``discarding'') elements
of a collection $\Ccal$ without affecting the Bredon homology of
the $\Ccal$-approximation is Lemma~\ref{lem:take-away-one} below.  (The
lemma will be promoted from a statement about $*$ to a statement about
a $G$-space $X$ in Section~\ref{section: pruning}.)  We need a little
more terminology to explain the criterion. If $W$ is a finite group,
then $W$ acts on $\pgroupsntof{W}$ (by conjugation), and there is a
natural map of spaces
\begin{equation}\label{equation:key-map}
\lvert\pgroupsntof{W}\rvert_{hW} \to *_{hW} = BW \,.
\end{equation}
If $M$ is a $W$-module, $\pgroupsntof{W}$ is called
\defining{homologically (resp. cohomologically) $M$-ample} if
\eqref{equation:key-map} induces an isomorphism on ordinary twisted
homology (resp. cohomology) with coefficients in~$M$.

For notation in the following lemma, suppose that 
$D$ is a subgroup of $G$,
let $N_G(D)$ denote the normalizer of $D$ in $G$,
and let $W_G(D)=N_G(D)/D$ be the Weyl group of $D$ in $G$. 
Notice that $W_G(D)$ acts on $G/D$ by $G$-equivariant maps.

\begin{takeawayone}
\takeawayonetext
\end{takeawayone}

To apply Lemma~\ref{lem:take-away-one}, a criterion for ampleness is
needed. It is known that if $M$ is a $\integers_{(p)}$-module, 
and $W$ has an element
of order $p$ that acts trivially on $M$, then $\pgroupsntof{W}$ is
$M$-ample 
(Proposition~\ref{proposition: algebraic pruning criterion}).
Typically, such elements come from the centralizer of $D$, which is
why condition~\eqref{condition: centralizer} is present in
Theorem~\ref{theorem: main theorem}.

If we start from a Bredon (co)homology isomorphism
$\left(\Pcal_n^\diamond\right)_{\pgroupsof{\Sigma_{n}}} \to
\Pcal_n^\diamond$, we would like to know how many subgroups must be
eliminated from $\pgroupsof{\Sigma_{n}}$ before we obtain an
identifiable calculation of the Bredon (co)homology
of~$\Pcal_{n}^{\diamond}$. In particular, how many subgroups must be
eliminated before we can conclude that $\Pcal_{n}^{\diamond}$ has the
Bredon (co)homology of a point?  Suppose that $X$ is a $G$-space and
that $\Ccal$ is a collection of subgroups of $G$ such that
$X_{\Ccal}\rightarrow X$ is an isomorphism on Bredon (co)homology.  If
it happens that $X^{H}\simeq *$ for all $H\in\Ccal$, then $X_{\Ccal}$
has the same Bredon (co)homology as a point.  Hence so does $X$,
and $X^{\diamond}$ has trivial reduced Bredon (co)homology.

The preceding paragraph suggests that if we start from a Bredon
(co)homology isomorphism
$\left(\Pcal_n^\diamond\right)_{\pgroupsof{\Sigma_{n}}} \to
\Pcal_n^\diamond$, 
it would be nice to discard from $\pgroupsof{\Sigma_{n}}$ the subgroups 
whose fixed point spaces are \emph{not} contractible.  We hope
that there are not too many of them, and that they can be discarded
from $\pgroupsof{\Sigma_{n}}$ without damaging the starting Bredon
(co)homology isomorphism

We call a subgroup $H\subseteq\Sigma_{n}$ \defining{problematic} if
$\left(\Pcal_n^\diamond\right)^H$ is not contractible. 

\begin{badsubgroups}
\badsubgroupstext
\end{badsubgroups}

This proposition tells us that in fact there are very few problematic
subgroups.  The proof of Theorem~\ref{theorem: main theorem} then goes by
using Lemma~\ref{lem:take-away-one} to eliminate these few subgroups from the
collection~$\pgroupsof{\Sigma_{n}}$.  We can usually
establish the ampleness required in the hypothesis of
Lemma~\ref{lem:take-away-one} by using centralizing elements, as
discussed just after the statement of the lemma, above. 
However, it turns out that in a few cases,
appropriate centralizing elements do not exist. These are cases in
which an isotropy group of $\Pcal_n$ contains a \defining{$p$-centric}
problematic subgroup. Nonetheless, it turns out that the ampleness hypothesis
holds in these cases, with one exception, 
because all the relevant homology and cohomology
groups vanish.  This is where we need condition~\eqref{condition:
  involution} of the theorem.

In the end, the only problematic subgroups that cannot be eliminated using
one of the methods we have described are elementary abelian $p$-subgroups
of $\Sigma_{n}$ that act transitively on~$\n$. This occurs only when 
$n=p^{k}$, and in this case the Tits building $B_{k}$ comes up because 
it is the fixed point space of the elementary abelian $p$-group 
$\Delta_{k}$ of~$\Sigma_{p^{k}}$ acting on~$\Pcal_{n}$. 

\subsubsection*{Organization}\mbox{}\\
\indent In Sections~\ref{section: background on Bredon homology}
and~\ref{section: mackey functors}, we give background
on Bredon homology and cohomology and Mackey functors, and we 
state the key properties that we use and prove some basic results. 
Section~\ref{section: background on approximations} reviews approximation
theory from \cite{A-D} and proves Proposition~\ref{proposition: family},
the initial approximation result for Bredon homology and cohomology.
Section~\ref{section: pruning} discusses how to discard subgroups from
an approximating collection.

Section~\ref{section: fixed point sets} shows that most $p$-subgroups
of $\Sigma_{n}$ have contractible fixed point spaces and 
Proposition~\ref{proposition: bad subgroups} identifies those
that may not (``problematic'' subgroups). 
Section~\ref{section: isotropy groups} collects some elementary 
results about the isotropy groups of~$\Pcal_{n}$. 
Then Section~\ref{section: centralizers and involutions} studies centralizers of 
problematic subgroups inside isotropy groups of~$\Pcal_{n}$,
with a view to acquiring the algebraic input for 
Proposition~\ref{proposition: algebraic pruning criterion}. 

Section~\ref{section: centralizers} establishes that the ampleness
hypothesis needed to use 
Proposition~\ref{proposition: pruning criterion} holds in the case of
the coefficients in Theorem~\ref{theorem: main theorem}. This gives
the data needed to prove Theorem~\ref{theorem: main theorem} and
Corollary~\ref{corollary:only-one-dimension} in
Section~\ref{section:approximation-results}.  Finally,
Section~\ref{section: our coefficients} looks at the coefficient
functors of Example~\ref{example: relevant-Mackey-functor}.

\subsubsection*{Notation and Terminology}
\mbox{}\\
\indent Throughout the paper, $G$ is a finite group and $p$ is 
a fixed prime. 
We use ``space'' to mean a simplicial set, and we
often do not distinguish between a category and its nerve, trusting to
context to indicate which is under
discussion. 
We fix a model for a free, contractible
$G$-space $EG$, and given a $G$-space~$X$, we write $X_{hG}$ for the
unreduced Borel construction $(EG\times X)/G$. If $X$ has a basepoint and
the basepoint is fixed by the $G$-action, we write $X_{\hobased G}$
for the reduced Borel construction $\left(EG_{+}\wedge X\right)/G$.

If $H$ is a subgroup of a group~$G$, we write $\Cen_{G}(H)$ for the 
centralizer of $H$ in~$G$, we write $N_{G}(H)$ for the normalizer of 
$H$ in~$G$, and we write $W_{G}(H)=N_{G}(H)/H$ for the Weyl group of
$H$ in~$G$. 
Following Quillen~\cite{Quillen}, we denote the poset of non-identity
$p$-subgroups of $G$ by $\pgroupsntof{G}$, or just $\pgroupsntnone$ if
the group is clear from context. The poset of all $p$-subgroups of
$G$, including the trivial group, is denoted by $\pgroupsof{G}$ 
or~$\pgroupsnone$. 
If $X$ is a $G$-space, we write $\Iso(X)$ to denote the collection
of subgroups of $G$ that appear as isotropy groups of~$X$.

We regard a partition $\lambda$ of $\n$ as corresponding to
equivalence classes of an equivalence relation, where
$x\sim_{\lambda}y$ means that $x$ and $y$ are in the same component
of~$\lambda$. (We write simply $x\sim y$ if the partition is clear
from context.)  We write $\class{\lambda}$ to denote the set of
components, or equivalence classes, of a partition~$\lambda$.

\subsubsection*{Acknowledgement} 
We thank Jesper Grodal for pointing us toward helpful
references, and for making comments that helped to improve the paper.
In particular, Grodal's promptings led us to formulate
Proposition~\ref{proposition: family} in terms of Mackey functors that
are projective relative to a family.

\section{Bredon homology and cohomology}
\label{section: background on Bredon homology}

Let $G$ be a finite group. In this section, we give general background
on $G$-spaces and on their Bredon homology and
cohomology~\cite{Bredon}.

We work simplicially. Thus, by a $G$-space $X$ we mean a simplicial
set with a $G$-action. A $G$-map $f: X\to Y$ is called a
\defining{$G$-equivalence} if it induces an equivalence $X^K\to Y^K$
of fixed point spaces for each subgroup $K$ of $G$.  The geometric
realization and the singular set functors preserve fixed points. It
follows that if $f\colon X \to Y$ is a $G$-equivalence of simplicial
sets, then the geometric realization of $f$ is an 
equivalence in the category of $G$-topological spaces.

Given a $G$-space $X$, let $\Iso(X)$ denote the set of all subgroups
of $G$ that appear as isotropy subgroups of simplices of $X$. The
following lemma gives an economical criterion for
recognizing $G$-equivalences. A proof can be found
in~\cite[4.1]{Homology}, but it is older than this.
 
\begin{lemma}  
\label{lemma:check-G-equivalence}
If $f: X\to Y$ is a map of $G$-spaces that induces equivalences
$X^K\to Y^K$ for each $K\in \Iso(X)\cup\Iso(Y)$, then $f$ is a
$G$-equivalence.
\end{lemma}

We next describe the Bredon chain and cochain complexes; to minimize
redundancy, we handle both simultaneously, with the terms for
cohomology in parentheses.  Let $\mu$ be an additive, covariant (resp.
contravariant) functor from finite $G$-sets to abelian groups, i.e., a
functor taking coproducts of $G$-sets to sums (resp. products) of
abelian groups. Such a functor will be called a homological (resp.
cohomological) coefficient system for $G$. The functor $\mu$ can be
extended to all $G$-sets by the formula
\[
\mu(T)\definedas\underset{T_\alpha}{\colim} \mu\left(T_\alpha\right)
\]
(resp. 
$\mu(T)\definedas \displaystyle\lim_{T_\alpha}\mu\left(T_\alpha\right)$),
where $T_\alpha$ ranges over the poset of finite $G$-subsets of $T$.  
The Bredon chains (resp. cochains) on $X$ with coefficients in $\mu$
are obtained by applying $\mu$ degreewise to $X$ to obtain a
simplicial (resp.  cosimplicial) abelian group, and then applying
Dold-Kan's normalized chains functor, to obtain a chain (resp.
cochain) complex.  The Bredon homology (resp. cohomology) of $X$ with
coefficients in $\mu$, denoted $\BredonHom{*}{G}{X}{\mu}$ (resp.
$\BredonCoh{*}{G}{X}{\mu}$), is the homology (resp. cohomology) of
this chain (resp. cochain) complex.

Let $\Ocal(G)$ be the orbit category of $G$, namely the category of
transitive $G$-sets and $G$-equivariant maps. If $\mu$ is any functor
from $\Ocal(G)$ to abelian groups then it can be extended to an
additive functor on all finite $G$-sets in a unique way. Thus, one may
define Bredon homology (resp. cohomology) with coefficients in an
arbitrary functor from $\Ocal(G)$ to abelian groups.

\begin{remark}\label{remark: Bredon Kan}
  One may view an additive covariant functor $\mu$ from $G$-sets to
  abelian groups as a functor from $\Ocal(G)$ to chain complexes that
  takes values in complexes concentrated in degree zero. Then the
  Bredon chain functor is the homotopy left Kan extension of $\gamma$
  from $\Ocal(G)$ to the category of simplicial $G$-sets. In the
  contravariant case, the Bredon cochain functor is the homotopy right Kan
  extension of $\mu$ from the opposite category of $\Ocal(G)$ to
  the opposite category of simplicial $G$-sets.
\end{remark}

The following example will come up later in the paper, in the proofs of
Lemma~\ref{lemma:one-subgroup-approximation} and, implicitly,
Lemma~\ref{lem:take-away-one}.
\begin{example}\label{example:bredon-ordinary}
  Let $G$ be a group, and $D\subset G$ a subgroup. Note that the Weyl
  group $W_{G}(D)$ acts on $G/D$ by $G$-equivariant
  maps. If $X$ is a space with an action of~$W_{G}(D)$, then
  $G\times_{N_{G}(D)} \left(X\times EW_{G}(D)\right)$ 
  is a space with an action of~$G$.
  Let $\mu$ be a (homological or cohomological) coefficient system for
  $G$. There is an isomorphism
%\Note{Check reduced versus unreduced.}
\begin{equation}  \label{eq: twisted homology}
\BredonHom{*}{G}{G\times_{N_{G}(D)} \left(X\times EW_{G}(D)\right)}{\mu}
          \cong H_*\left(X_{hW_{G}(D)}; \mu(G/D)\right)
\end{equation}
or, in the cohomological case
\begin{equation}  \label{eq: twisted cohomology}
\BredonCoh{*}{G}{G\times_{N_{G}(D)} \left(X\times EW_{G}(D)\right)}{\mu}
          \cong H^*\left(X_{hW_{G}(D)}; \mu(G/D)\right).
%\BredonCohRed{*}{G}{G\times_N (X\times EW)}{\mu}
%          \cong H^*\left(X_{hW}; \mu(G/D)\right).
\end{equation}
Here the groups on the left are Bredon homology or cohomology groups,
while the groups on the right are ordinary homology or cohomology with
twisted coefficients in the $W$-module~$\mu(G/D)$. To see where these
isomorphisms come from, let $\mu$ be a homological coefficient
system for $G$, and let $S$ be a set with a free $W$-action. There
is an isomorphism 
$\mu\left(G\times_N S\right) 
    \cong \mu(G/D)\otimes_{\integers[W]} \integers[S]$,
which is natural in $S$. If $\mu$ is cohomological, then there is 
a natural isomorphism 
$\mu\left(G\times_N S\right) 
     \cong \Hom_{\integers[W]}\left(\integers[S], \mu(G/D)\right)$. 
It follows that there is an isomorphism of simplicial abelian groups
\[
\mu\left(G\times_N (X\times EW)\right) 
           \cong \mu(G/D)\otimes_{\integers[W]} \integers[X\times EW].
\] 
This isomorphism implies the isomorphism in~\eqref{eq: twisted homology}. 
The cohomological version of~\eqref{eq: twisted cohomology}
is proved similarly.
\end{example}

\subsubsection*{Homotopy properties}
Bredon homology has good formal properties. Proofs of the homological
cases of the following two well-known lemmas are given
in~\cite[4.8]{Sharp} and~\cite[4.11]{Sharp}. The cohomological
versions can be proved similarly. The first one also follows from
Remark~\ref{remark: Bredon Kan}.

\begin{lemma}  \label{lemma: Bredon invariance}
  If $f: X\rightarrow Y$ is a $G$-equivalence, then $f$ induces
  isomorphisms on Bredon homology and cohomology (with any
  coefficients).
\end{lemma}

\begin{lemma} \label{lemma: Bredon gluing}
A homotopy pushout square 
\[
\begin{CD}
X' @>>> X\\
@VVV @VVV \\
Y'@>>> Y
\end{CD}
\] 
of $G$-spaces gives long exact sequences
in Bredon homology and cohomology (with any coefficients).
\end{lemma}

\begin{remark*}
  A homotopy pushout square of $G$-spaces is a square that, upon
  taking fixed points $({-})^K$ for any subgroup $K\subset G$, becomes
a homotopy pushout square of spaces.
\end{remark*}

\subsubsection*{Restriction of coefficients} 
We recall that Bredon homology and cohomology have 
restriction of coefficients. If $\mu$ is a covariant or
contravariant functor and $K$ is a subgroup of $G$, we can define 
a coefficient functor $\mu\vert_K$ from  $K$-sets to abelian groups given by
\[
  \mu\vert_K(S) = \mu(G\times_K S)\,.
\]
Then for any $K$-space $Y$, depending on the variance of $\mu$, we have
\begin{equation}
\label{equation: restriction-of-coefficients}
\begin{aligned}
\BredonHom{*}{K}{Y}{\mu\vert_K}
        &= \BredonHom{*}{G}{G\times_{K} Y}{\mu}\\
\BredonCoh{*}{K}{Y}{\mu\vert_K}
        &= \BredonCoh{*}{G}{G\times_{K} Y}{\mu}\,.
\end{aligned}
\end{equation}
\subsubsection*{Reduced Bredon homology} 
For a pointed $G$-space $(X, *)$ and a covariant coefficient system
$\mu$, there is a split monomorphism from the Bredon chains on $*$
to the Bredon chains on $X$. The homology groups of the quotient
complex are the reduced Bredon homology groups of $X$ with
coefficients in~$\mu$. They are denoted by
$\BredonHomRed{*}{G}{X}{\mu}$. Similarly, if $\mu$ is a
contravariant coefficient system, there is a split epimorphism from
the cochains on $X$ to the cochains on $*$, and the cohomology groups
of the kernel are the reduced Bredon cohomology groups of $X$, denoted
$\BredonCohRed{*}{G}{X}{\mu}$.

There is an isomorphism of Bredon homology groups
\[
\BredonHom{*}{G}{X}{\mu}
      \cong\BredonHomRed{*}{G}{X}{\mu}\oplus\BredonHom{*}{G}{*}{\mu}
\]
and a similar one for cohomology groups. There
are analogues of Lemmas~\ref{lemma: Bredon invariance}
and~\ref{lemma: Bredon gluing} for reduced Bredon homology and
cohomology. It follows that if $X$ is equivariantly contractible, then
the reduced Bredon homology and cohomology groups of $X$ vanish with
any coefficients.

\section{Mackey functors} 
\label{section: mackey functors}

To obtain our results, we need to work with Bredon homology theories
that behave well with respect to approximation by $p$-subgroups. It
turns out that the key property required is the presence of
transfers for finite covers of $G$-spaces.  To obtain well-behaved
transfers, we use coefficient functors that extend to Mackey functors. 
The first part of this section collects background
on Mackey functors; the second part discusses 
projectivity relative to a collection of subgroups, 
a key hypothesis for Theorem~\ref{theorem: main theorem}.

There are several equivalent definitions of Mackey functors in the
literature. We will follow the treatments of Dress~\cite{Dress} and
Webb~\cite{WebbSplit, Webb}.

\begin{definition} \label{definition: Mackey functor} 
  A \defining{Mackey functor (for $G$)} is a pair of additive functors
  $\coeff=(\gamma, \gamma^\natural)$ from the category of finite
  $G$-sets to abelian groups, satisfying the following conditions.
\begin{enumerate}
\item The functor $\gamma$ is covariant and the
  functor $\gamma^\natural$ is contravariant. 
\item The functors $\gamma$ and $\gamma^\natural$ agree on objects.
  Thus for every finite $G$-set $S$, we write
  $\coeff(S)\definedas\gamma(S)=\gamma^\natural(S)$.
\item \label{naturality} 
  If the diagram on the left below is a pullback diagram of $G$-sets,
  then the diagram on the right commutes.
\[
  \begin{CD}
     X      @>{a}>>  U     \\
    @V{f}VV         @V{g}VV\\
     Y      @>{b}>> V
  \end{CD}
\qquad\qquad\quad
  \begin{CD}
     \coeff(X) @> \gamma(a) >> \coeff(U)\\
    @A\gamma^\natural(f)AA   @A{\gamma^\natural(g)}AA\\
     \coeff(Y)@> \gamma(b) >> \coeff(V)
  \end{CD}
 \]
\end{enumerate}
\end{definition}

Because a Mackey functor has both a covariant and a contravariant
part, it provides coefficient systems for both Bredon homology and
Bredon cohomology. We denote the resulting homology and cohomology
groups by $\BredonHom{*}{G}{-}{\coeff}$ and
$\BredonCoh{*}{G}{-}{\coeff}$.

The following remark concerns extending a Mackey functor from finite
$G$-sets to arbitrary $G$-sets. Although we will only be interested in
Bredon homology of simplicial sets of finite type, we include the
remark for completeness.
\begin{remark}\label{remark: extension}
  There are two ways to extend a Mackey functor $\coeff=(\gamma,
  \gamma^\natural)$ from finite set to arbitrary sets. The first
  is to define
\[
\gamma^\natural(T)=\gamma(T)
         \definedas\underset{T_\alpha}{\colim} \gamma(T_\alpha)
\]
where $T_\alpha$ ranges over finite subsets of $T$. With this
definition, it is clear that $\gamma$ is functorial with respect to
all $G$-maps between $G$-sets. The contravariant functor
$\gamma^\natural$ is also functorial, but with respect to
\emph{finite} covers of $G$-sets.  Indeed, suppose $f\colon X \to Y$
is a finite cover of $G$-sets. Then $f^{-1}(-)$ defines a poset map
from finite subsets of $Y$ to finite subsets of $X$. One uses this to
define a map $\gamma^\natural(Y)\to \gamma^\natural (X)$. With this
definition, condition~\eqref{naturality} of
Definition~\ref{definition: Mackey functor} holds for an arbitrary
square diagram of $G$-sets, provided $f$ and $g$ are finite covers.

The other way to extend $\coeff$ is by the formula
\[
\gamma(T)=\gamma^\natural(T)
     \definedas\underset{T_\alpha}{\lim\,} \gamma^\natural(T_\alpha).
\]
With this definition, $\gamma^\natural$ remains contravariantly
functorial with respect to arbitrary $G$-maps, but
$\gamma$ is only functorial with respect to finite covers.
Condition~\eqref{naturality} of 
Definition~\ref{definition: Mackey functor} holds for an arbitrary
square diagram of $G$-sets, provided $a$ and $b$ (rather than $f$
and~$g$) are finite covers.
\end{remark}

For the remainder of this section we discuss projectivity of a Mackey
functor with respect
to a collection of subgroups.  Let $\coeff=(\gamma, \gamma^\natural)$
be a Mackey functor for $G$ and let $Z$ be a finite $G$-set. Then one
may define a new Mackey functor $\coeff_Z$ by the formula
$\coeff_Z(T)=\coeff(Z\times T)$.
Moreover, if $f\colon Z\to Y$ is an equivariant map of finite
$G$-sets, then $\gamma$ and $\gamma^\natural$ induce natural
transformations of Mackey functors $\coeff_Z\to \coeff_Y$ and
$\coeff_Y\to \coeff_Z$, respectively. In particular, taking $Y=*$ we
obtain natural transformations of Mackey functors $\theta_Z\colon
\coeff_Z \to \coeff$ and $\theta^Z\colon\coeff\to \coeff_Z$.

\begin{definition}
  A Mackey functor $\coeff$ is \defining{projective relative to $Z$} if
  $\theta_Z$ is a split surjection of Mackey functors.
\end{definition}

Recall that a \defining{collection} is a (necessarily finite) set of
subgroups of $G$ that is closed under conjugation.

\begin{definition} 
  Let $G$ be a finite group, let $\Ccal$ be a collection of subgroups
  of $G$, and let $\coeff$ be a Mackey functor for~$G$. We say that
  $\coeff$ is \defining{projective relative to $\Ccal$} if $\coeff$ is
  projective relative to $Z=\coprod G/H$, where $H$ ranges
  over a set of representatives of conjugacy classes of elements of
  $\Ccal$.
\end{definition}

The following routine lemma and corollary allow easier verification
that a Mackey functor is projective relative to a
collection~$\Ccal$, in particular relative to $\pgroupsof{G}$, the
collection of all $p$-subgroups of $G$.
(See Lemma~3.2 of~\cite{Webb}.)

\begin{lemma}
  A Mackey functor is projective relative to $\Ccal$ if
  and only if it is projective relative to the collection of maximal
  elements of~$\Ccal$.
\end{lemma}

\begin{corollary}
  Let $G$ be a finite group and let $P$ be a $p$-Sylow subgroup of
  $G$.  A Mackey functor is projective relative to $\pgroupsof{G}$ if
  and only if the natural map $\theta_{G/P}\colon \coeff_{G/P}\to
  \coeff$ is a split surjection of Mackey functors.
\end{corollary}

The following definition and lemma allow us to recognize Mackey
functors that are projective relative to~$\pgroupsof{G}$. In practice,
all of our examples satisfy this condition. 

\begin{definition} \label{definition: p-constrained} 
  Suppose that $\coeff$ is a Mackey functor for $G$. We say that
  $\coeff$ \defining{has the $p$-transfer property} if for every $G$-set
  $Z$ whose cardinality is prime to $p$, the composition
\[
\theta_{Z}\circ\theta^{Z}\colon \coeff {\to} \coeff_{Z} {\to}{\coeff}
\]    
is an isomorphism from $\coeff$ to itself.
\end{definition}

\begin{lemma}\label{lemma:p-transfer}
  If a Mackey functor $\coeff$ has the $p$-transfer property, then it
  takes values in $\integers_{(p)}$-modules, and it is projective
  relative to~$\pgroupsnone$.
\end{lemma}

\begin{proof}
  For the first assertion, let $Z$ be a set with trivial $G$-action. 
  The composed map 
  $\theta_{Z}\circ\theta^{Z}\colon 
        \coeff\longrightarrow\coeff_{Z}\longrightarrow\coeff$ 
  is multiplication by the cardinality of~$Z$. If we assume that the
  composition is an isomorphism for every $Z$ of cardinality prime
  to~$p$, it means exactly that $\coeff$ takes values in 
  $\integers_{(p)}$-modules.

  For the second assertion, take $Z=G/P$, where $P$ is a $p$-Sylow
  subgroup of $G$. Then 
  $\theta_{G/P}\circ\theta^{G/P}\colon \coeff\longrightarrow
 \coeff_{G/P}\longrightarrow{\coeff}$ is an isomorphism of Mackey
  functors, which implies that $\theta_{G/P}$ is a split surjection of
  Mackey functors. Hence $\coeff$ is projective relative to $p$-groups.
\end{proof}

\section{Approximations}
\label{section: background on approximations}

In this section we develop general tools for approximating a $G$-space
$X$ by other $G$-spaces whose Bredon homology or cohomology may be
easier to calculate.
First, we recall the notion of $\Ccal$-approximation used in~\cite{A-D} and
give a sufficient condition for this approximation to induce an
isomorphism on Bredon homology and cohomology
(Definition~\ref{definition: approximation}, 
Lemma~\ref{lemma: check isotropy groups}).  
Second, we give an explicit model for an approximation that involves only 
one conjugacy class of subgroups 
(Lemma~\ref{lemma:one-subgroup-approximation}).
Lastly, we observe that if $\coeff$ is a Mackey functor that is
projective relative to a \defining{family} $\Fcal$, then
$\Fcal$-approximation induces an isomorphism on Bredon homology and
cohomology with coefficients in $\coeff$. This allows us to use the
family of all $p$-subgroups of $G$ as a canonical jumping-off point
for approximations.  In Section~\ref{section: pruning}, we build on
this beginning and set up an inductive process to reduce the size of
the controlling collection from the family of all $p$-subgroups to a
manageable collection, without changing the Bredon homology or
cohomology.

Recall that for a $G$-space~$X$, we write $\Iso(X)$ for the collection
of subgroups of $G$ that appear as isotropy groups of~$X$. Let $\Ccal$
be a collection of subgroups of $G$. A $G$-space $X$ is said to have
\defining{$\Ccal$-isotropy} if $\Iso(X)\subseteq\Ccal$. A $G$-map
$X\to Y$ is a \defining{$\Ccal$-equivalence} if $f^K: X^K\to Y^K$ is
an equivalence for each $K\in\Ccal$.

\begin{definition}   \label{definition: approximation}
  Given a $G$-space $X$, a \defining{$\Ccal$-approximation} to $X$ is
  a $\Ccal$-equivalence $X_{\Ccal}\to X$ such that $X_{\Ccal}$ has
  $\Ccal$-isotropy.
\end{definition}

We review the construction of a functorial $\Ccal$-approximation,
which goes back to Elmendorf~\cite{Elmendorf}. Let $\Ocal_\Ccal$ be
the full subcategory of transitive $G$-sets and $G$-equivariant maps
whose objects have isotropy groups only in $\Ccal$, and let $i$ denote
the inclusion functor from $\Ocal_\Ccal$ to $G$-spaces.  The
$\Ccal$-approximation functor is the endofunctor of $G$-spaces
obtained by taking the identity functor, restricting it to
$\Ocal_\Ccal$, and then taking homotopy left Kan extension back to the
category of $G$-spaces.  In more concrete terms, $X_\Ccal$ can be
constructed as the homotopy coend 
$i \otimes_{\Ocal_\Ccal} \map_G(-, X)$. Here $\map_G(-, X)$ is the
contravariant functor from $\Ocal_\Ccal$ to spaces given by 
$S\mapsto \map_G(S, X)$. For more detail, see~\cite[4.8]{Homology}
and~\cite[Section 3]{A-D}.

It follows from the construction of $X_{\Ccal}$ that if $X$ 
has finite type, then $X_{\Ccal}$ does too. The functoriality of the
construction, together with Lemma~\ref{lemma:check-G-equivalence},
implies that $\Ccal$-approximations are unique up to a canonical zigzag of
$G$-equivalences. Further, $\Ccal$-approximation commutes with
homotopy pushouts of $G$-spaces.  The Mayer-Vietoris property
(Lemma~\ref{lemma: Bredon gluing}) thus implies that in order to
determine whether $X_{\Ccal}\rightarrow X$ induces an isomorphism on
Bredon homology or cohomology, it is enough to check this condition
for the orbits used in building~$X$.

\begin{lemma}       \cite[3.2 and 3.3]{A-D}
\label{lemma: check isotropy groups}
Suppose that $X$ is a $G$-space.  If, for all $K\in\Iso(X)$, 
the map $(G/K)_{\Ccal}\to G/K$ gives an
isomorphism on Bredon homology with coefficients in $\gamma$
(resp. cohomology with coefficients in $\gamma^\natural$), then
$X_{\Ccal}\to X$ gives such an isomorphism as well.
\end{lemma}

Since $G/K\cong G\times_{K} *$, Lemma~\ref{lemma: check isotropy groups}
involves approximating a space induced up from a subgroup. 
There is a general lemma available for this purpose. 
If $K$ is a subgroup of $G$, consider the collection
\[
\Ccal\downto K \definedas \{H\,\vert\, H\in\Ccal \mbox{ and } H\subseteq K\}\,.
\]
The following elementary lemma, used with $Y=*$,  
is helpful in applying Lemma~\ref{lemma: check isotropy groups}.

\begin{lemma}\label{lemma:reduce-to-subgroup} \cite[2.12]{A-D}
  Let $K$ be a subgroup of $G$ and let $Y$ be a $K$-space. Then there is a
  canonical $G$-equivalence 
  $G\times_K \left(Y_{\Ccal\downto K}\right)\simeq
  \left(G\times_K Y\right)_{\Ccal}$.
\end{lemma}

Next, we give an explicit description of the approximation for
a collection consisting of a single subgroup of $G$ together with 
its conjugates. We also record the Bredon homology and cohomology of the
approximation.
Let $D$ be a subgroup of $G$, and let $\Dcal$ be the collection
consisting of all conjugates of $D$ in $G$. Recall that the normalizer
$N_{G}(D)$ acts on the fixed point space $X^D$,
and also on the free contractible $W_{G}(D)$-space~$EW_{G}(D)$

\begin{lemma}\label{lemma:one-subgroup-approximation}
  With the above notation, let $F=X^D$. Then
  for any $G$-space $X$, the natural map
\[
  Y \definedas G\times_{N_{G}(D)} \left(EW_{G}(D)\times F\right) \to X
\]
  is a $\Dcal$-approximation $X_{\Dcal}\to X$. For any Bredon
  coefficient systems $\gamma$ and $\gamma^\natural$, there are natural
  isomorphisms
\begin{equation}  \label{equation: twisted groups}
    \begin{aligned}
    \BredonHom{*}{G}{X_{\Dcal}}{\gamma} 
             &\cong H_*(F_{hW_{G}(D)};\gamma(G/D))\\
    \BredonCoh{*}{G}{X_{\Dcal}}{\gamma^\natural} 
             &\cong H^*(F_{hW_{G}(D)};\gamma^\natural(G/D))\, 
    \end{aligned}
\end{equation}
\end{lemma}

The (co)homology groups on the right side of 
\eqref{equation: twisted groups} are the ordinary local coefficient
(co)homology groups associated to the natural action of $W$
on the coefficients through the action of $W$ on~$G/D$.

\begin{proof}[Proof of Lemma~\ref{lemma:one-subgroup-approximation}] 
  The verification that $Y\simeq X_{\Dcal}$ proceeds by checking the
  two conditions that characterize $X_{\Dcal}$. First, the orbit
  types: all isotropy groups of $Y$ are indeed conjugate to~$D$.
  Second, the fixed point spaces: $Y^D$ is $EW_{G}(D)\times F$, which is
  homotopy equivalent to~$X^D$, as required of~$X_{\Dcal}$. The homology and
  cohomology calculations then follow by
  Example~\ref{example:bredon-ordinary}.
\end{proof}

To close this section, we establish a canonical starting point for
approximation calculations when the Bredon coefficient system is
projective relative to a family. The case $X=*$ of the following
proposition can be read off as a special case 
of~\cite[Theorem A]{WebbSplit}. The case when $\Fcal$ is the family of
all $p$-subgroups of~$G$ is essentially~\cite[6.4]{Sharp}.

\begin{proposition}
 \label{proposition: family} 
\allpgroupstext
\end{proposition}

\begin{proof}
We first assert that if $H\in\Fcal$, then 
\[
G/H\times X_{\Fcal}   \to G/H \times X
\]
is actually a $G$-equivalence by Lemma~\ref{lemma:check-G-equivalence}.
This is because if $K$ is an isotropy group of either the source or the target,
then $K$ is conjugate to a subgroup of~$H$. Hence~$K\in\Fcal$, and
$\left(X_{\Fcal}\right)^{K} \to X^{K}$ is a homotopy equivalence
by definition of~$X_{\Fcal}$.

  Let $Z=\coprod G/H$, where $H$ ranges over a set of representatives
  of conjugacy classes in $\Fcal$. Since our Mackey functor
  $\coeff=(\gamma, \gamma^\natural)$ is projective 
  relative to $Z$, the homomorphisms induced on Bredon chains and
  cochains by the map $X_\Fcal \to X$ are retracts of the
  corresponding homomorphisms induced by the map $Z\times X_{\Fcal}\to
  Z\times X$. However, the latter map is a $G$-equivalence, and so 
  induces isomorphisms on Bredon homology and cohomology. 
\end{proof}

\section{Approximations controlled by smaller 
collections}
\label{section: pruning}

Recall that $\pgroupsof{G}$ is the family of all $p$-subgroups of a
fixed group~$G$. Let $X$ be a $G$-space. It follows from
Proposition~\ref{proposition: family} that if $\coeff$ is a Mackey
functor projective relative to $\pgroupsof{G}$, then the approximation
map $X_{\pgroupsof{G}}\rightarrow X$ induces isomorphisms on Bredon
homology and cohomology with coefficients in $\coeff$.
 
If $\Ccal\subseteq\pgroupsof{G}$ is a subcollection, then there is a
natural factoring $X_{\Ccal}\rightarrow X_{\pgroupsof{G}}\rightarrow X$.
We can ask if $X_{\Ccal}\rightarrow X_{\pgroupsof{G}}$ induces
isomorphisms on Bredon homology and cohomology with coefficients in~$\coeff$,
in which case the map $X_{\Ccal}\rightarrow X$ induces such
isomorphisms as well.
Our main result along these lines is 
Proposition~\ref{proposition: pruning criterion} below, which is 
an essential ingredient in the proof of Theorem~\ref{theorem: main theorem}.

Let $W$ be a group. Recall that $\pgroupsntof{W}$ is the poset of 
non-identity $p$-subgroups of $W$. Let $\lvert\pgroupsntof{W}\rvert$
be the nerve of this poset. Let $M$ be a $W$-module. 
Recall that $\pgroupsntof{W}$ is said to be homologically $M$-ample if the map
$\lvert\pgroupsntof{W}\rvert\to *$ induces an isomorphism on homology
with twisted coefficients in~$M$: 
\[
 H_*(\lvert\pgroupsntof{W}\rvert_{hW};M) \xrightarrow{\ \cong\ } H_*(BW;M).\\ 
\]
Similarly, we say that $\pgroupsntof{W}$ is cohomologically $M$-ample if
we have an isomorphism on cohomology with twisted coefficients in~$M$:
\[
 H^*(\lvert\pgroupsntof{W}\rvert_{hW};M) \xleftarrow{\ \cong\ } H^*(BW;M).
\]

The following lemma gives a useful criterion for when a group can be
removed from a collection without affecting Bredon homology. It is
similar to~\cite[Proposition 9.3]{Sharp}.  Note that for a subgroup
$D$ of~$G$, the Weyl group $W_G(D)$ acts by $G$-maps on the set $G/D$,
so $\gamma(G/D)$ is a $W_G(D)$-module.

\begin{lemma}\label{lem:take-away-one}
\takeawayonetext
\end{lemma}

The proof will appear after an auxiliary lemma.
If $\Ccal$ is a collection of subgroups of $G$, a subcollection
$\Dcal\subset\Ccal$ is \defining{initial} if whenever $D\in\Dcal$ and
$C\in\Ccal$ with $C\subset D$, then $C\in \Dcal$. Note that the union
of two collections is a collection.

\begin{lemma}
\label{lemma:fix-up-d}
Let $\Ccal$ be a collection and let $\Dcal\subset \Ccal$ be an initial
subcollection. Let $Y\to Z$ be a map of $G$-spaces that induces an
equivalence $Y_{\Ccal\setminus \Dcal}\to Z_{\Ccal\setminus \Dcal}$.
Then there is a homotopy pushout diagram of $G$-spaces
\begin{equation*}
\begin{CD}
       Y_{\Dcal} @>>> Y_{\Ccal} \\
         @VVV            @VVV   \\
       Z_{\Dcal} @>>> Z_{\Ccal}
\end{CD}\,.
\end{equation*}
\end{lemma}

\begin{proof}
  It is only necessary to check that for each subgroup $H\in\Ccal$,
  the indicated diagram becomes an ordinary homotopy pushout diagram
  upon taking $H$-fixed points, which is clear.  (Note that the spaces on
  the left have empty $H$-fixed sets for $H\in\Ccal\setminus \Dcal$,
  because $\Dcal$ is initial.) 
\end{proof}

\begin{proof}[Proof of Lemma~\ref{lem:take-away-one}]
The diagram of Lemma~\ref{lemma:fix-up-d} with 
$Y=\left(*\right)_{\Ccal\setminus \Dcal}$ and $Z=*$ gives the homotopy pushout diagram
\[
\begin{CD}
\left(*_{\Ccal\setminus\Dcal}\right)_{\Dcal} 
                           @>>> \left(*_{\Ccal\setminus\Dcal}\right)_{\Ccal} \\
         @VVV            @VVV   \\
       (*)_{\Dcal} @>>> (*)_{\Ccal}
\end{CD}\,.
\]
In the upper right corner, $\left(*_{\Ccal\setminus\Dcal}\right)_{\Ccal}$
is $G$-equivalent to $(*)_{\Ccal\setminus \Dcal}$ 
(Lemma~\ref{lemma:check-G-equivalence}).
Further, $\Dcal$ consists of just one conjugacy class, allowing us to use
Lemma~\ref{lemma:one-subgroup-approximation} to obtain 
explicit formulas for $\left(*_{\Ccal\setminus\Dcal}\right)_{\Dcal}$ and 
$\left(*\right)_{\Dcal}$. The result is the homotopy pushout diagram 
of $G$-spaces
\[
\begin{CD}
G\times_N\left(EW\times \left(*_{\Ccal\setminus\Dcal}\right)^{D}\right) @>>> (*)_{\Ccal\setminus\Dcal}\\
    @VVV                       @VVV\\
G\times_ N EW         @>>>     (*)_{\Ccal}
\end{CD}\,.
\]
In view of the Mayer-Vietoris property 
(Lemma~\ref{lemma: Bredon gluing}), the right vertical map is an
isomorphism on Bredon homology or cohomology if and only if the left
map is. By the second part of
Lemma~\ref{lemma:one-subgroup-approximation}, the Bredon (co)homology
of spaces on the left reduces to ordinary (co)homology with twisted
coefficients. The only remaining point to note is that the fixed point
set $((*)_{\Ccal\setminus\Dcal})^D$ is homotopy equivalent to
$\lvert\pgroupsntof{W}\rvert$ via a $W$-equivariant map.
This is proved in the third paragraph of~\cite[Pf.~of~8.3]{Homology}.
\end{proof}

\begin{corollary}\label{cor:prunealot}
  Let $\mu$ be a homological or cohomological coefficient system
  for~$G$. Let $\Ccal$ be a collection of $p$-subgroups of $G$ that is
  closed under passage to $p$-supergroups. Suppose that for each
  $p$-subgroup $D$ of $G$ with $D\notin \Ccal$, the collection
  $\pgroupsntof{W_G(D)}$ is $\mu(G/D)$-ample. Then the map
  $*_{\Ccal}\to *_{\pgroupsof{G}}$ induces an isomorphism on Bredon
  homology or cohomology with coefficients in~$\mu$.
\end{corollary}
\begin{proof}
  Use Lemma~\ref{lem:take-away-one} repeatedly to eliminate the
  conjugacy classes of elements of $\pgroupsof{W_G(D)}\setminus \Ccal$
  one by one, in such a way that smaller groups are removed before
  larger ones. In this way, all intermediate collections are closed
  under passage to $p$-supergroups, and at each step one removes a
  minimal element of the collection. Thus
  Lemma~\ref{lem:take-away-one} applies at each step.
\end{proof}
The following proposition is the main result of this section. Note that if $D\subset K \subset G$ then $W_K(D)$ acts on $G/D=G\times_K K/D$ by $G$-equivariant maps.
\begin{proposition}  \label{proposition: pruning criterion}
\pruningcriteriontext
\end{proposition}

\begin{proof}
  By Lemma~\ref{lemma: check isotropy groups}, it is enough to check that
  for each $K\in \Iso(X)$ the map $(G/K)_{\Ccal}\to (G/K)_{\pgroupsnone}$ 
  gives an isomorphism on Bredon homology (or cohomology).  
  By Lemma~\ref{lemma:reduce-to-subgroup}, we need for
$G\times_{K}\left(*\right)_{\Ccal\downarrow K}
  \rightarrow G\times_{K}\left(*\right)_{\pgroupsnone\downarrow K}$ 
to induce an isomorphism. By definition of restriction of 
coefficients~\eqref{equation: restriction-of-coefficients}, 
this amounts to showing that 
$\left(*\right)_{\Ccal\downarrow K} \rightarrow 
         \left(*\right)_{\pgroupsnone\downarrow K}$ 
induces an isomorphism on Bredon homology or cohomology with coefficients 
in $\mu\restrict_K$. This follows by 
Corollary~\ref{cor:prunealot} with $K$ playing the role of $G$.
\end{proof}

There remains the question of how to establish ampleness.  The
following criterion is more or less standard and well known. 

\begin{proposition}
\label{proposition: algebraic pruning criterion}
\algebraicpruningtext
\end{proposition}

\begin{proof}
  The homology case when $M$ is an $\field_p$-module
  is~\cite[6.3]{Homology}. The cohomology case follows from
  Grodal~\cite[Corollary 5.4 and Theorem 9.1]{Grodal}. It is clear
  that Grodal's methods can be adapted to include the homological case
  as well. We will give a direct proof starting from the
  $\field_p$-module case.
   
  Let $\mykernel$ be the kernel of the action map $W\to\Aut(M)$.
  By \cite[6.3]{Homology}, the natural twisted coefficient
  homology map
\begin{equation*}  
H_{*}\left(\lvert\pgroupsntof{W}\rvert_{hW};\field_p[W/\mykernel]\right)
\rightarrow
H_{*}\left(BW;\field_p[W/\mykernel]\right)
\end{equation*}
is an isomorphism. By Shapiro's lemma, this is the same as the natural
map
\[
   H_*(\lvert\pgroupsntof{W}\rvert_{h\mykernel};\field_p)\to H_*(B\mykernel;\field_{p})\,.
\]
Since $\lvert\pgroupsntof{W}\rvert_{h\mykernel}$ and $B\mykernel$ are of finite type, we
conclude that the map $\lvert\pgroupsntof{W}\rvert\to *$ induces an isomorphism on homology
\begin{equation}\label{eq: mu iso}
H_*(\lvert\pgroupsntof{W}\rvert_{h\mykernel};\integers_{(p)})
           \xrightarrow{\cong} H_*(B\mykernel;\integers_{(p)})\,
\end{equation}
and a similar isomorphism on cohomology.
The proof is finished by comparing the Serre spectral sequences
for the following diagram, whose rows are fibrations:
\[
\begin{CD}
\lvert\pgroupsntof{W}\rvert_{h\mykernel} @>>> \lvert\pgroupsntof{W}\rvert_{hW} @>>> B(W/\mykernel)\\
@VVV @VVV @VVV\\
B\mykernel @>>> BW @>>> B(W/\mykernel) 
\end{CD}\,.
\]
The abutments of these two spectral sequences are the homology or
cohomology of the total spaces with twisted coefficients in~$M$, and
by~\eqref{eq: mu iso}, the vertical maps induce isomorphisms on the
$E^{2}$-pages. The proposition follows.
\end{proof}

\section{Fixed point spaces of $\Pcal_{n}$}
\label{section: fixed point sets}

In this section, we study the fixed point spaces of $p$-subgroups of
$\Sigma_{n}$ acting on $\Pcal_{n}$. To motivate this problem in the
current context, suppose that $X$ is a $G$-space and that $\Ccal$ is a
collection of subgroups of $G$ for which $X_{\Ccal}\rightarrow X$ is
known to be an isomorphism on Bredon (co)homology. If it happens
that $X^{H}\simeq *$ for all $H\in\Ccal$, then $X_{\Ccal}$ has
the $G$-equivariant homotopy type of a point. A Bredon
(co)homology isomorphism $X_{\Ccal}\rightarrow X$ would say that $X$
has the Bredon (co)homology of a point as well. Hence elements of~$\Ccal$
that have \emph{non}-contractible fixed point spaces on $X$ can be
considered obstructions to $X$ having the same Bredon 
(co)homology as a point. 

To apply this idea, recall that $\pgroupsof{\Sigma_{n}}$ denotes the family
of all $p$-subgroups of~$\Sigma_{n}$, and suppose $M$ is a Mackey
functor for $\Sigma_{n}$ that is projective relative
to~$\pgroupsof{\Sigma_{n}}$.  By Proposition~\ref{proposition: family},
\[
\left(\Pcal_{n}\right)_{\pgroupsof{\Sigma_{n}}}\to \Pcal_{n}
\]
induces an isomorphism on Bredon (co)homology with coefficients
in~$M$. Our goal in this section is to show that very few
$p$-subgroups of $\Sigma_{n}$ have non-contractible fixed point spaces on 
$\Pcal_{n}$. There are therefore very few obstructions to $\Pcal_{n}$
having the Bredon (co)homology of a point. 

\begin{definition}
If $H\subseteq\Sigma_{n}$, we say that $H$ is \defining{problematic} if
$\left(\Pcal_{n}\right)^{H}$ is not contractible. 
\end{definition}

Our main result in this section is the following proposition. 
The proof appears at the end of the section.

\begin{proposition} \label{proposition: bad subgroups} 
\badsubgroupstext
\end{proposition}

As a result, we conclude that very few subgroups of $\Sigma_{n}$ are
problematic. In fact, if $p^i\mid n$, then there is a unique (up to
conjugacy) elementary abelian $p$-subgroup of rank~$i$ in~$\Sigma_{n}$
that acts freely on~$\n$.  We study the centralizers of these few
problematic subgroups inside isotropy groups of $\Pcal_{n}$ in
Section~\ref{section: centralizers and involutions}, for the purpose of
eliminating them from the approximating collection.

For the proof of the key lemma below, we need a little more notation.
If $V\subseteq\Sigma_{n}$, let
$\left(\Pcal_{n}\right)_{\orb{V}}$ denote the poset of proper nontrivial
partitions of $\n$ whose classes are unions of $V$-orbits. Given a
partition $\lambda$ that is stabilized by the action of~$V$, let
$\orbit{\lambda}{V}$ denote the coarsening of $\lambda$ obtained by 
merging classes of $\lambda$ that contain elements in the same orbit of $V$.
Explicitly, $x\sim_{\orbit{\lambda}{V}}y$ if there exists $v\in V$
such that $x\sim_{\lambda}vy$. As a result, note that the equivalence 
classes of $\orbit{\lambda}{V}$ are unions of orbits of~$V$. 

\begin{lemma} \label{lemma: categorical retraction} 
  Let $H$ be a $p$-subgroup of $\Sigma_{n}$. Suppose that there exists
  a nontrivial subgroup $V$ of the center of $H$ with the property
  that for all proper partitions $\lambda$ that are fixed by ${H}$,
  the partition $\orbit{\lambda}{V}$ is proper. Then the nerve of
  $\left(\Pcal_{n}\right)^{H}$ is contractible.
\end{lemma}

\begin{proof}
Consider the inclusion into $\left(\Pcal_{n}\right)^{H}$ of the subposet 
whose objects are unions of $V$-orbits: 
\begin{equation}    \label{equation: intersection inclusion}
\left(\Pcal_{n}\right)_{\orb{V}}\cap \left(\Pcal_{n}\right)^{H}
\longrightarrow 
\left(\Pcal_{n}\right)^{H}. 
\end{equation}
We assert that this inclusion has a left adjoint. Indeed, if
$V$ is central in~$H$, then any partition $\lambda$ that is stabilized
by $H$ has the property that $\orbit{\lambda}{V}$ is also stabilized by~$H$. 
Provided that $\orbit{\lambda}{V}$ is always proper, a routine check 
shows that the functor
$\lambda\mapsto\orbit{\lambda}{V}$ is left adjoint to the 
inclusion~\eqref{equation: intersection inclusion}. 
Hence~\eqref{equation: intersection inclusion}
induces an equivalence on nerves. 
However, the left side has an initial object, namely
the partition of $\n$ by the orbits of~$V$, which is fixed by $H$
because $V$ is central in~$H$. Therefore
$\left(\Pcal_{n}\right)_{\orb{V}}\cap \left(\Pcal_{n}\right)^{H}$
has contractible nerve, which finishes the proof.  
\end{proof}

\begin{remark*}
Interpreting Lemma~\ref{lemma: categorical retraction}
in the case $n=p$ is slightly pedantic. The only $p$-subgroup of 
$\Sigma_{p}$ is $H=\integers/p$, and the available candidate for $V$ is 
$H$ itself. One proper 
partition is fixed by $V$, namely the discrete partition, 
but $V$ acts transitively on its classes. 
Hence the hypothesis of Lemma~\ref{lemma: categorical retraction}
is not satisfied, and we do not conclude that
$\left(\Pcal_{p}\right)^{H}$ is contractible. And, indeed,
$\left(\Pcal_{p}\right)^{H}$ is the empty set. 
\end{remark*}

To prove that $H$ is elementary abelian in
Proposition~\ref{proposition: bad subgroups}, we need a little group
theory. Given a $p$-group $H$, let $\Phi(H)$ be the Frattini subgroup
of~$H$, i.e., $\Phi(H)$ is the normal subgroup generated by commutators and
$p$-th powers.  Note that any homomorphism from $H$ to an elementary
abelian $p$-group factors through the quotient group~$H/\Phi(H)$. The
following lemma is standard.

\begin{lemma}    \label{lemma: kernel}
If $H$ is a $p$-group and is not elementary abelian, 
then there exists a subgroup $V\subseteq \Phi(H)$
such that $V$ has order~$p$ and is contained in the center of $H$. 
\end{lemma} 

\begin{proof}
  If $H$ is not elementary abelian, then $\Phi(H)$ is
  a nontrivial normal subgroup of~$H$, which necessarily has
  nontrivial intersection with the center because $H$ is a $p$-group. 
  We can then pick out an element of order~$p$ to generate~$V$.
\end{proof}

The following proposition now gives the group-theoretic structure
in Proposition~\ref{proposition: bad subgroups}. 

\begin{proposition}   \label{proposition: bad is elementary abelian}
If $H$ is a problematic $p$-subgroup of $\Sigma_{n}$, then $H$
is elementary abelian. 
\end{proposition}

\begin{proof}
  We prove the contrapositive. Suppose that $H$ is not elementary
  abelian, and let $V\subseteq \Phi(H)$ be the subgroup provided by
  Lemma~\ref{lemma: kernel}.  We want to apply 
  Lemma~\ref{lemma: categorical retraction} to show that $H$ is not
  problematic.

  Suppose $\lambda\in\left(\Pcal_{n}\right)^{H}$; we need to show that
  $\orbit{\lambda}{V}$ is proper. If 
  $\orbit{\lambda}{V}$ fails to be proper, then
  $V$ acts transitively on the equivalence classes of $\lambda$.  In
  this case, since $\lambda$ has more than one equivalence class, it
  must have exactly $p$ equivalence classes.  The action of $H$
  permutes the equivalence classes of $\lambda$, giving a homomorphism
  $H\rightarrow\Sigma_{p}$. However, $H\rightarrow\Sigma_{p}$ 
 necessarily factors
  through~$H/\Phi(H)$, because $H$ is a $p$-group and the only
  $p$-subgroups of $\Sigma_{p}$ are elementary abelian.  Since 
  $V\subseteq\Phi(H)$, this says that $V$ acts trivially on the
  classes of~$\lambda$, a contradiction. 

  We have established that for any $\lambda\in\left(\Pcal_{n}\right)^{H}$, 
  the partition $\orbit{\lambda}{V}$ is proper.  By 
  Lemma~\ref{lemma: categorical retraction},
  $\left(\Pcal_{n}\right)^{H}$ is contractible.
\end{proof}

\begin{proof}[Proof of Proposition~\ref{proposition: bad subgroups}]
  If $H$ is problematic, then we already know from
  Proposition~\ref{proposition: bad is elementary abelian} that $H$ is
  an elementary abelian $p$-group. We need to show that any element
  $h\in H$ acts freely on~$\n$. Let $V\cong\integers/p$ be the
  subgroup generated by~$h$. Because $H$ is problematic, the
  contrapositive of Lemma~\ref{lemma: categorical retraction} tells us
  that there exists a partition $\lambda$ in
  $\left(\Pcal_{n}\right)^{H}$ such that $V$ acts transitively on the
  equivalence classes of~$\lambda$. Since $\lambda$ is not the
  indiscrete partition, $\lambda$ must have exactly $p$ equivalence
  classes, freely permuted by $V$, which therefore acts freely on
  $\n$.
\end{proof}

\section{Isotropy in $\Pcal_{n}$}
\label{section: isotropy groups}
In this section we collect some elementary results about the isotropy
groups of the $\Sigma_{n}$-space~$\Pcal_{n}$. These results will be
used to produce centralizing elements of problematic subgroups of
isotropy groups.
 
We begin by setting our conventions regarding wreath products, because
these groups figure prominently in the isotropy groups of~$\Pcal_{n}$.
Suppose that $G$ is a group acting (on the left) on the set $\s$ and
that $H$ is any group. The wreath product $H\wr G$ is the semi-direct
product $H^{\s}\rtimes G$:
\[
1\longrightarrow H^{\s}\longrightarrow H\wr G\longrightarrow G
         \longrightarrow 1.
\]
We denote a general element of this group
by $(h_1, \ldots, h_{s}; g)$, and the group law is given by the
formula
\[
\left(h_1, \ldots, h_{s}; g\right)\cdot \left(h_1', \ldots, h_{s}'; g'\right)
          =\left(h_{1} h_{g^{-1}(1)}', \ldots,h_{s}h_{g^{-1}(s)}'; gg'\right).
\]
Accordingly, the formula for the inverse of an element is:
\[
\left(h_1, \ldots, h_{s}; g\right)^{-1}
      =\left(h_{g(1)}^{-1}, \ldots, h_{g(s)}^{-1}; g^{-1}\right).
\]
There is a natural group monomorphism 
\begin{equation}     \label{eq: diagonal inclusion}
\begin{array}{rcl}
H\times G  & \longrightarrow &H\wr G    \\
(h,g)      & \mapsto         &(h,\ldots, h; g). 
\end{array}
\end{equation}

\begin{definition}   \label{defn: special groups}
  We write $\Diag(H)$ for the image of $H\times\{e\}$ in $H\wr G$
  under~\eqref{eq: diagonal inclusion}.  If $T$ is a subgroup of~$G$,
  we write $\Twiggle$ for the subgroup of $H\wr G$ given by the image
  of $\{e\}\times T$ under~\eqref{eq: diagonal inclusion}.
\end{definition}

In particular, $\Gwiggle$ is the image of~$\{e\}\times G$
under~\eqref{eq: diagonal inclusion}. 
Note that $\Diag(H)$ and $\Gwiggle$ centralize each other in $H\wr G$. 

From Section~\ref{section: pruning}, we know that eliminating a
subgroup from an approximating collection requires thinking about
normalizers, and we begin with an elementary calculation for 
a special case of a normalizer in a wreath product. 

\begin{lemma}\label{lemma: wreath normalizer}
  Suppose that a subgroup $T\subseteq G$ acts transitively on~$\s$, and
let $N=N_{G}(T)$. Then
  the normalizer of $\Twiggle$ in $H\wr G$ is $\Diag(H)\times\Nwiggle$. 
\end{lemma}

\begin{proof}
An element $(h_1, \dots, h_s; g)\in H\wr G$ normalizes $\Twiggle$ 
if and only if for every $t\in T$ there exists $t'\in T$ such that
\[
\left(h_1, \dots, h_s; g\right)^{-1} \cdot \left(e, \ldots, e; t\right) 
     \cdot \left(h_1, \dots, h_s; g\right) 
              =(e, \ldots, e; t'). 
\]
Evaluating the left-hand side, we obtain
\begin{align*}
\left(h_{g(1)}^{-1}, \dots, h_{g(s)}^{-1}\,;\right. & \left. g^{-1}\right) 
 \cdot\left(e, \ldots, e\,; t\right) \cdot \left(h_1, \dots, h_s\,; g\right) \\ 
&= \left(h_{g(1)}^{-1}, \dots, h_{g(s)}^{-1}\,; g^{-1}t\right)  
    \cdot \left(h_1, \dots, h_s\,; g\right)\\
&= \left(h_{g(1)}^{-1}\,h_{t^{-1}g(1)}, 
            \,\dots\, , h_{g(s)}^{-1}\,h_{t^{-1}g(s)}\,; g^{-1}tg\right).
\end{align*}
This calculation allows us to verify that elements of
$\Diag(H)\times \Nwiggle$ normalize $\Twiggle$ in $H\wr G$, because 
if $\left(h_1, \dots, h_s; g\right)=(h,\dots,h;\,g)$ where
$g\in N_{G}(T)$, then 
the last line reduces to $\left(e,\dots,e\,;\,g^{-1}tg\right)$
with $g^{-1}tg\in T$. 

To see that a normalizing element of $\Twiggle$ must be in 
$\Diag(H)\times \Nwiggle$, observe that the calculation above implies that 
$\left(h_1, \dots, h_s; g\right)$ normalizes $\Twiggle$ only
if $g^{-1}tg\in T$ for all $t\in T$, so we must have $g\in N_{G}(T)$. Further,
$\left(h_1, \dots, h_s; g\right)$ normalizes $T$ only if for every
$i\in \s$ and $t\in T$, we have $h_{g(i)}=h_{t^{-1}g(i)}$. Since the
action of $T$ on $\s$ is transitive, it follows that
if $\left(h_1, \dots, h_s; g\right)$ normalizes $T$, then 
$h_1=h_2=\cdots=h_s$, so $\left(h_1, \dots, h_s\right)\in\Diag(H)$. 
The lemma follows. 
\end{proof}
 
Next we review the standard action of the wreath product on a 
product set. As above, suppose that $G$ acts on the set $\s$, and 
suppose also that $H$ acts on the set $\r$.
Then the wreath product $H\wr G$ acts on $\r\times \s$ as follows.
If $(i,j)\in \r\times\s$, then
 \[
 (h_1, \ldots, h_s; g)(i,j)=(h_{g(j)}(i), g(j)).
 \]
If we visualize $\r\times \s$ as $s$ columns, each containing the set~$\r$,
then $(h_1, \ldots, h_s; g)$ acts by first letting $g$ permute the set of
columns, and then (for each $i\in\s$) acting by $h_{i}$ on the $i$-th column. 
This action preserves the partition of $\r\times \s$ defined by the columns,
that is, by pre-images of points of $\s$ under the projection map 
\begin{equation}   \label{eq: product}
p\colon \r\times \s \to \s,
\end{equation}
and in fact $\Sigma_{\r}\wr \Sigma_{\s}$ is the full isotropy group of 
this partition. 
 
With these preliminaries in hand, we turn to the isotropy subgroups of
simplices of~$\Pcal_{n}$ under the action of~$\Sigma_{n}$. The
zero-dimensional simplices are partitions of $\n$. If $\lambda$ is
such a partition, we denote its isotropy group~$\isogroupof{\lambda}$.
By definition, elements of $\isogroupof{\lambda}$ are bijective maps
$\sigma\colon \n\to \n$ such that $x\sim_{\lambda}y$ if and only if
$\sigma x\sim_{\lambda}\sigma y$ for all~$x, y\in \n$.  Note that the
action of $\isogroupof{\lambda}\subseteq\Sigma_{n}$ on $\n$ induces
an action of $\isogroupof{\lambda}$ on~$\class{\lambda}$, the set of
equivalence classes (or ``components'') of~$\lambda$.

There is a special type of partition that will play an important role.
\begin{definition}
We say that a partition $\lambda$ of $\n$ is \defining{regular} if
the elements of $\class{\lambda}$ all have the same cardinality. 
\end{definition}

Suppose $\lambda$ is a regular partition of $\n$, with classes of
cardinality~$r$. We fix a class $\cclass\in\class{\lambda}$, and we
choose bijections between $\cclass$ and each of the other elements
of~$\class{\lambda}$.  These choices define a (non-canonical)
bijection and corresponding isomorphism:
\begin{align}   \label{eq: product}
\n \longleftrightarrow \cclass\times\class{\lambda}\\
\strut\isogroupof{\lambda}\cong\Sigma_{\cclass}\wr\Sigma_{\class{\lambda}}.
\end{align}
The composition of~\eqref{eq: product} with projection to $\class{\lambda}$
gives a map that is $\isogroupof{\lambda}$-equivariant: 
\[
\n\xrightarrow{\ \cong\ }\cbold\times\class{\lambda}
\xrightarrow{\ \, p \,\ }\class{\lambda}.
\]

Generalizing to partitions that are not regular, suppose that
$\lambda$ has classes of various cardinalities $r_1 \ldots, r_l$,
where $r_1 \ldots, r_l$ are pairwise distinct.  If there are $s_i$
classes of cardinality~$r_i$, then $n=r_{1}s_{1}+\cdots + r_{l}s_{l}$, and
$\isogroupof{\lambda}$ is (non-canonically) isomorphic to the
following product of wreath products:
\[
\left(\Sigma_{r_1}\wr \Sigma_{s_1}\right)
         \times\cdots\times 
         \left(\Sigma_{r_l}\wr\Sigma_{s_l}\right).
\]
From this formula, we get the following lemma. 

\begin{lemma}   \label{lemma: regular if transitive}
The action of $\isogroupof{\lambda}$ on $\class{\lambda}$ 
is transitive if and only if $\lambda$ is regular. 
\end{lemma}

\begin{proof}
If $\lambda$ has $s_i$ classes of cardinality~$r_i$, as above, then
the action of $\isogroupof{\lambda}$ on $\class{\lambda}$ factors 
through $\Sigma_{s_{1}}\times\dots\times\Sigma_{s_{l}}$. The 
action of the latter is only transitive if and only if $l=1$, 
i.e., $\lambda$ is regular. 
\end{proof}

Moving on to isotropy groups of higher-dimensional simplices of
$\Pcal_n$, a typical non-degenerate simplex is a chain
\[
\Lambda=\left(\lambda_{0}<\dots<\lambda_{j}\right)
\]
where $\lambda_{0},\dots,\lambda_{j}$ are partitions of~$\n$. 
We write $\isogroupof{\Lambda}$ for the subgroup of~$\Sigma_{n}$
that stabilizes all of the partitions $\lambda_{0},\dots,\lambda_{j}$, 
that is, $\isogroupof{\Lambda}$ is the isotropy group of $\Lambda$
as a simplex of the $\Sigma_{n}$-space~$\Pcal_{n}$. 
For each $i$, we know 
$\isogroupof{\Lambda}\subseteq\isogroupof{\lambda_{i}}$,
so there is an action of $\isogroupof{\Lambda}$ 
on $\class{\lambda_{i}}$ for each~$i$.

We need a few elementary notions for chains. First, 
the evident notions of restriction and isomorphism. 

\begin{definition}\mbox{}
\begin{enumerate}
\item If $\lambda$ is a partition of~$A$ and $X\subseteq A$, we
  define $\lambda\restrict_{X}$, the \defining{restriction of
    $\lambda$ to $X$}, as the partition of $X$ obtained by intersecting
  each equivalence class of $\lambda$ with $X$.
\item  If 
 $\Lambda=\left(\lambda_{0}\leq\dots\leq\lambda_{j}\right)$ is a chain
of partitions of~$A$, we write $\Lambda\restrict_{X}$ for the chain
$
\Lambda\restrict_{X}
  =\left(\lambda_{0}\restrict_{X}\leq\dots\leq\lambda_{j}\restrict_{X}\right).
$

\item Suppose given sets $A$ and $A'$, and chains of partitions
$\Lambda=\left(\lambda_{0}\leq\dots\leq\lambda_{j}\right)$ and
$\Lambda'=\left(\lambda'_{0}\leq\dots\leq\lambda'_{j}\right)$
of $A$ and~$A'$, respectively. 
We say that $\Lambda$ and $\Lambda'$
  are \defining{isomorphic} if there exists a bijection
  $f:A\rightarrow A'$ such that for all $i$ with $0\leq i\leq j$, we
  have $x\sim_{\lambda_{i}}y$ $\iff$ $f(x)\sim_{\lambda'_{i}}f(y)$.
\end{enumerate}
\end{definition}

Note that the restriction of a partition $\lambda$ can be discrete or
indiscrete even if $\lambda$ itself is neither. Likewise, the 
restriction of a strict inequality of partitions can be
an equality, so a nondegenerate chain 
(all strict inequalities) may become degenerate upon restriction.

In Section~\ref{section: centralizers and involutions}, subgroups 
of $\Sigma_{n}$ that have transitive actions
on the set of classes of a partition are of particular interest. 
To state the next lemma, which is the first step to understanding such 
situations, we need a little more notation. If 
$\Lambda= \left(\lambda_{0}<\lambda_{1}<...<\lambda_{j}\right)$
and $i\leq j$, we define the shorter chain of partitions
$\Lambda_{<\,i}=\left(\lambda_{0}<\lambda_{1}<...<\lambda_{i-1}\right)$.  
If $\cclass\in\class{\lambda_{i}}$, then we write 
$\left(\Lambda_{<\,i}\right)\!\restrict_{\,\cclass}$ for the restriction of 
$\Lambda_{<\,i}$ to~$\cclass$, and we write $K(\cclass)$ for the subgroup of
the symmetric group on $\cclass$ that stabilizes
$\left(\Lambda_{<\,i}\right)\!\restrict_{\,\cclass}$.

\begin{lemma}\label{lemma: transitive}
Let 
$\Lambda= \left(\lambda_{0}<\lambda_{1}<...<\lambda_{j}\right)$
be a chain of partitions of~$\n$ and let 
$\isogroupof{\Lambda}\subseteq\Sigma_{n}$ be its
isotropy group. Let $i$ be an integer with $0\le i\le j$. 
Then the action of $\isogroupof{\Lambda}$ on
$\class{\lambda_i}$ is transitive if and only if the following
two conditions hold:
\begin{enumerate}
\item \label{item: regularity assumption}
The partitions $\lambda_{i},\lambda_{i+1},\dots,\lambda_{j}$ are all regular. 
\item \label{item: isomorphism assumption}
For any $\cclass,\cclass'\in\class{\lambda_{i}}$, the chains
$\left(\Lambda_{<\,i}\right)\!\restrict_{\,\cclass}$ 
and $\left(\Lambda_{<\,i}\right)\!\restrict_{\,\cclass'}$ 
are isomorphic. 
\end{enumerate}
\end{lemma}

\begin{example}\label{example: chain}
Let $n=18$. Consider the following chain of partitions
$\Lambda=\left(\lambda_0< \lambda_1< \lambda_2\right)$:
\[\begin{small}
\begin{array}{cccc}
\!\!\lambda_2:
     & \{1,2,3,4,5,6\}
     &\!\!\{7,8,9,10,11,12\}
     &\!\!\{13,14,15,16,17,18\} \\[5pt]
\!\!\lambda_1:
     & \{1,2,3\}      \ \ \{4,5,6\}
     &\!\!\{7,8,9\}   \ \ \{10,11,12\}
     &\!\!\{13,14,15\}\ \ \{16,17,18\} \\[5pt]
\!\!\lambda_0:
     &\{1,2\}\{3\}   \ \{4,5\}\{6\}
     &\!\!\{7,8\}\{9\}   \ \{10,11\}\{12\}
     &\!\!\{13,14\}\{15\}\ \ \{16,17\}\{18\}
\end{array}
\end{small}
\]

We verify assumptions~\eqref{item: regularity assumption} 
and~\eqref{item: isomorphism assumption} of 
Lemma~\ref{lemma: transitive} for~$i=1$. 
First, the partitions $\lambda_{2}$ and $\lambda_{1}$ are regular, so 
\eqref{item: regularity assumption} is satisfied for~$i=1$. 
Second, $\Lambda_{<\,1}$ is just~$\lambda_{0}$, and the restriction of 
$\lambda_{0}$ to any class $\cclass\in\class{\lambda_{1}}$ consists of
a singleton and a two-element set. Hence restricting
$\Lambda_{<\,1}$ to elements of $\class{\lambda_{1}}$ gives
pairwise isomorphic partitions. According to Lemma~\ref{lemma: transitive}, 
the action of $\isogroupof{\Lambda}$ on $\class{\lambda_{1}}$ is transitive. 

And indeed, by inspection we have
$\isogroupof{\Lambda}
    \cong\left(\Sigma_{2}\times\Sigma_{1}\right)\,\wr\,
      \left(\Sigma_{2}\,\wr\,\Sigma_{3}\right)$. 
The action of $\isogroupof{\Lambda}$ on $\class{\lambda_{1}}$ 
is given by $\left(\Sigma_{2}\,\wr\,\Sigma_{3}\right)
       \subseteq\Sigma_{6}\cong\Sigma_{\class{\lambda_{1}}}$,
which is transitive. (See also Example~\ref{example: continued}.) 
\smallskip
\end{example}

\begin{proof}[Proof of Lemma~\ref{lemma: transitive}]

  Suppose $\isogroupof{\Lambda}$ acts transitively on
  $\class{\lambda_i}$. Then $\isogroupof{\Lambda}$ also acts
  transitively on the set of classes of
  $\lambda_{i+1},\dots,\lambda_{j}$, and hence these partitions are
  regular by Lemma~\ref{lemma: regular if transitive}.  Further,
  suppose that $\cclass,\cclass'\in\class{\lambda_i}$. The transitive
  action of $\isogroupof{\Lambda}$ on $\class{\lambda_i}$ gives an
  element $\sigma\in\isogroupof{\Lambda}$ taking $\cclass$
  to~$\cclass'$; then $\sigma$ induces the required isomorphism from
  $\left(\Lambda_{<\,i}\right)\!\restrict_{\,\cclass}$ to
  $\left(\Lambda_{<\,i}\right)\!\restrict_{\,\cclass'}$.

  Now let us prove the converse. 
First, we assert that for any $t$
  with $i\leq t\leq j$ and for any
  $\cclass,\cclass'\in\class{\lambda_{t}}$, we have isomorphic chains
  $\left(\Lambda_{<\,t}\right)\!\restrict_{\,\cclass}$ and
  $\left(\Lambda_{<\,t}\right)\!\restrict_{\,\cclass'}$. The point is
  that while assumption~\eqref{item: isomorphism assumption} only
  gives us such an isomorphism for~$t=i$, we can reach
  the same conclusion inductively for $t>i$ by using regularity of
  $\lambda_{i}$, $\lambda_{i+1}$,$\dots$,$\lambda_{j}$ 
  (assumption~\eqref{item: regularity assumption}). 

 To simplify notation for isotropy groups in what follows, 
  we write $\isogroupof{t}$ for the isotropy group of the chain
  $\left(\lambda_{0}<\lambda_{1}<...<\lambda_{t}\right)$.  
  We know 
  $\isogroupof{i}\supset\isogroupof{i+1}\dots
  \supset\isogroupof{j}=\isogroupof{\Lambda}$, 
  and our goal is to prove that $\isogroupof{\Lambda}$ acts
  transitively on~$\class{\lambda_{i}}$.  Our strategy is to first
  prove that $\isogroupof{i}$ acts transitively
  on~$\class{\lambda_{i}}$, and then we prove by induction on $t$ that
  $\isogroupof{t}$ likewise acts transitively on~$\class{\lambda_{i}}$
  for $i\leq t\leq j$.

  For the inductive hypothesis, we need to assume more about
  $\isogroupof{t-1}$ (for $t>i$) than simply a transitive action
  on~$\class{\lambda_{i}}$ (so we will need to verify the assumption
  explicitly for $K_{i}$ to get the base case). To state the hypothesis 
for $\isogroupof{t-1}$, let
  $\cclass,\cclass'$ be arbitrary elements of~$\class{\lambda_{i}}$,
  and let $C,C'\in\class{\lambda_{t-1}}$ be the unique classes such
  that $\cclass\subseteq C$ and $\cclass'\subseteq C'$.  We assume for
  the inductive hypothesis on $\isogroupof{t-1}$ that for any choice
  of $\cclass, \cclass'$, there exists an element
  $\sigma_{t-1}(\cclass,\cclass')\in\isogroupof{t-1}$ that is a
  bijection from $\cclass$ to~$\cclass'$, is a bijection from $C$ to
  $C'$, and is the identity on the complement of $C\sqcup\,C'$
  in~$\n$.

  To construct $\sigma_{i}(\cclass,\cclass')\in\isogroupof{i}$, the
  base case, recall that by assumption~\eqref{item: isomorphism
    assumption}, the chains
  $\left(\Lambda_{<\,i}\right)\!\restrict_{\,\cclass}$ and
  $\left(\Lambda_{<\,i}\right)\!\restrict_{\,\cclass'}$ are
  isomorphic.  This means there is a bijection $f:
  \cclass\rightarrow\cclass'$ that induces an isomorphism of
  $\left(\Lambda_{<\,i}\right)\!\restrict_{\,\cclass}$ to
  $\left(\Lambda_{<\,i}\right)\!\restrict_{\,\cclass'}$. We define
  $\sigma_{i}(\cclass,\cclass')\in\isogroupof{i}$ as $f$ on~$\cclass$,
  as $f^{-1}$ on $\cclass'$, and as the identity map on the complement
  of $\cclass\sqcup\cclass'$.

  For the inductive step, we need to construct
  $\sigma_{t}(\cclass,\cclass')\in\isogroupof{t}$.  Let
  $D,D'\in\class{\lambda_{t}}$ be the classes containing
  $\cclass,\cclass'$, respectively. Because $\lambda_{t-1}$ and
  $\lambda_{t}$ are both regular, $D$ and $D'$ are both constructed by
  merging the same number of classes of $\lambda_{t-1}$, say
  $D=C_{1}\sqcup\dots\sqcup\,C_{k}$ and
  $D'=C'_{1}\sqcup\dots\sqcup\,C'_{k}$ for classes
  $C_{1},\dots,C_{k},C'_{1},\dots,C'_{k}$ of~$\lambda_{t-1}$. We know
  from what was proved above that the restrictions of $\Lambda_{<\,t}$
  to any of the classes $C_{1},\dots,C_{k},C'_{1},\dots,C'_{k}$ are
  isomorphic.

   Now suppose that $\cclass\subseteq C_{1}$ and $\cclass'\subseteq C'_{1}$. We
   construct $\sigma_{t}(\cclass,\cclass')\in\isogroupof{t}$ as follows. On the
   complement of $D\sqcup D'$ in~$\n$, the element $\sigma_{t}(\cclass,\cclass')$
   acts by the identity.  
   On~$C_{1}\subseteq D$, use the bijection to $C'_{1}\subseteq D'$
   given by the inductive hypothesis.  On the other classes,
   $C_{2}$,\dots,$C_{k}$, use bijections to $C'_{2}$,\dots, $C'_{k}$,
   respectively, that are isomorphisms of the restrictions of
   $\Lambda_{<\,t}$ to each class.  Likewise, on $C'_{1}$,\dots,
   $C'_{k}$ use the inverses of those isomorphisms.  This completes
   the inductive step and finishes the proof.
\end{proof}

If, as in Lemma~\ref{lemma: transitive}, the isotropy subgroup of a
chain $\Lambda$ acts transitively on the equivalence classes of 
some $\lambda_{i}$ in the chain, then there is an explicit expression for the
isotropy group $\isogroupof{\Lambda}$ as a wreath product. To describe
this, we need another definition.

\begin{definition}
  Let $\mu$ and $\lambda$ be partitions of a set, with
  $\mu\leq\lambda$.  We define the 
\defining{partition of $\class{\mu}$ induced by~$\lambda$}, 
denoted $\lambdabar_{\mu}$, or simply $\lambdabar$ if $\mu$
is clear from context, by following equivalance relation:
 if $\cclass, \cclass'$ are 
equivalence classes of $\mu$, then 
\[
\cclass\sim_{\lambdabar_{\mu}}\,\cclass' \iff 
     \cclass\sqcup \cclass' \text{ is contained in a single equivalence class
  of } \lambda.
\] 
\end{definition}

For instance, in Example~\ref{example: chain}, we could consider 
$\lambda_{1}<\lambda_{2}$. Then $\lambdabar_2$, the partition of 
$\class{\lambda_{1}}$ induced by $\lambda_{2}$, is a partition of 
a six-element set into three subsets, each containing two elements:
\[
\lambdabar_2=\left\{
\begin{array}{c}
\{\{1,2,3\},\{4,5,6\}\}\\
\hspace{1.5em}\{\{7,8,9\},\{10,11,12\}\}\\
\hspace{.4em}\{\{13,14,15\},\{16,17,18\}\}.
\end{array}
\right.
\]

In the following lemma, and also in 
Section~\ref{section: centralizers and involutions}, 
we need a convention regarding
the extremes in a chain of partitions
$\Lambda= \left(\lambda_{0}<\lambda_{1}<...<\lambda_{j}\right)$. 
We adopt the convention that $\lambda_{-1}$ is the partition into 
singleton subsets (the discrete partition). 
Similarly, we interpret $\lambda_{j+1}$ as the partition containing
just one equivalence class (the indiscrete partition). 

\begin{lemma}\label{lemma: transitive gives wreath}
  Suppose that 
  $\Lambda= \left(\lambda_{0}<\lambda_{1}<...<\lambda_{j}\right)$, and
  suppose that $\isogroupof{\Lambda}$ acts transitively
  on~$\class{\lambda_{i}}$ for some $0\le i \le j$.  
  Let  $\Lambdabar_{>i}=\left(\lambdabar_{i+1}<\dots<\lambdabar_{j}\right)$
  be the induced chain of partitions of $\class{\lambda_{i}}$, and let
$K\left(\Lambdabar_{>i}\right)$ be the isotropy subgroup of $\Lambdabar_{>i}$
in~$\Sigma_{\class{\lambda_{i}}}$. Fix a class $\cclass\in\class{\lambda_{i}}$
and let $K(\cclass)\subseteq\Sigma_{\cclass}$ denote the isotropy 
group of 
$\left(\Lambda_{<\,i}\right)\!\restrict_{\,\cclass}$.
Then there is an isomorphism
\begin{equation}  \label{eq: iso to wreath}
\isogroupof{\Lambda}\cong K(\cclass)\wr
     K\left(\Lambdabar_{>i}\right). 
\end{equation}
\end{lemma}

\begin{proof}
We have a fixed class $\cclass\in\class{\lambda_{i}}$. By 
Lemma~\ref{lemma: transitive}, we know that for all 
$\cclass'\in\class{\lambda_{i}}$, the chains
$\left(\Lambda_{<\,i}\right)\!\restrict_{\,\cclass}$ 
and $\left(\Lambda_{<\,i}\right)\!\restrict_{\,\cclass'}$ 
are isomorphic. Hence for each $\cclass'\in\class{\lambda_{i}}$,
we can choose a bijection 
$\cclass\rightarrow\cclass'$ that respects the finer partitions
$\lambda_{0},$~\dots~$,\lambda_{i-1}$. Assembling these bijections
gives a bijection
\begin{equation}  \label{eq: n as product}
\n\longrightarrow\cclass\times\class{\lambda_{i}}.
\end{equation}
Under~\eqref{eq: n as product}, the chain $\Lambda$ of partitions of $\n$ 
corresponds to a chain of partitions of $\cclass\times\class{\lambda_{i}}$ 
as described below. The lemma will follow by identifying bijections from
$\cclass\times\class{\lambda_{i}}$ to itself that preserve the chain of 
partitions. 

We start in the middle of the chain. Subsets of $\n$ that are 
equivalence classes of
$\lambda_{i}$ correspond under \eqref{eq: n as product} precisely to
the columns of $\cclass\times\class{\lambda_{i}}$. Explicitly, if
$\cclass'\subseteq\n$ is an equivalence class of $\lambda_{i}$, then
the bijection $\cclass\rightarrow\cclass'$ used to define 
\eqref{eq: n as product} provides the correspondence between
the column $\{\left(x,\cclass'\right): x\in\cclass\}$ and 
$\cclass'\subseteq\n$. As indicated just after~\eqref{eq: product}, 
the isotropy group of $\lambda_{i}$ as a partition of 
$\cclass\times\class{\lambda_{i}}$ is then identified
by \eqref{eq: n as product} with 
$\Sigma_{\cclass}\wr\Sigma_{\class{\lambda_{i}}}$.  

Partitions finer than $\lambda_{i}$ take place within the columns of 
$\cclass\times\class{\lambda_{i}}$. To be specific, if $t<i$, we have
$\left(x,\cclass'\right)\sim_{\lambda_{t}}\left(y,\cclass''\right)$ if
and only if $\cclass'=\cclass''$ and $x\sim_{\lambda_{t}}y$.
We know that for any $\cclass'\in\class{\lambda_{i}}$, the bijection
$\cclass\rightarrow\cclass'$ induces an isomorphism from 
$\left(\Lambda_{<\,i}\right)\!\restrict_{\,\cclass'}$ to
$\left(\Lambda_{<\,i}\right)\!\restrict_{\,\cclass}$, and hence 
an isomorphism $K(\cclass)\xrightarrow{\cong} K(\cclass')$. 
As a consequence, the isotropy group of 
$\Lambda_{<\,i+1}$ regarded as a chain of partitions of 
$\cclass\times\class{\lambda_{i}}$ is isomorphic to 
$K(\cclass)\wr\Sigma_{\class{\lambda_{i}}}$.  

Finally, we note that coarsenings of~$\lambda_{i}$ are in one-to-one
correspondence with partitions of~$\class{\lambda_{i}}$. An element
$\sigma\in\Sigma_{n}$ stabilizes the coarsenings $\lambda_{i+1}<...<\lambda_{j}$
of $\lambda_{i}$ 
if and only if the image of $\sigma$ in $\Sigma_{\class{\lambda_{i}}}$ 
stabilizes~$\Lambdabar_{>i}$. We conclude that 
under~\eqref{eq: n as product}, the isotropy subgroup 
$\isogroupof{\Lambda}\subseteq\Sigma_{n}$ corresponds to 
$\isogroupof{\Lambda}\cong K(\cclass)\wr
     K\left(\Lambdabar_{>i}\right)$. 
\end{proof}

\begin{example}   \label{example: continued}
We continue with the setup of
Example~\ref{example: chain}. The action of
$\isogroupof{\Lambda}$ on $\class{\lambda_{1}}$ is transitive,
so we consider Lemma~\ref{lemma: transitive gives wreath}
applied with~$i=1$. 
The chain $\Lambdabar_{>1}$ consists simply of $\lambdabar_{2}$, which,
as mentioned just before Lemma~\ref{lemma: transitive gives wreath},
partitions the six-element set $\class{\lambda_{1}}$ into three
two-element subsets. Therefore
\[
K\left(\Lambdabar_{>1}\right)\cong \Sigma_{2}\,\wr\,\Sigma_{3}.
\]
Choose $\cclass=\{1,2,3\}\in\class{\lambda_{1}}$. Then $\lambda_{1}$
restricted to $\cclass$ is a singleton and a two-element set, so
$K(\cclass)=\Sigma_{1}\times\Sigma_{2}$. Thus the result of the
Lemma~\ref{lemma: transitive gives wreath} is what we found before 
by inspection:
\[
\isogroupof{\Lambda}\cong K(\cclass)\wr
     K\left(\Lambdabar_{>1}\right)
    \cong\left(\Sigma_{2}\times\Sigma_{1}\right)\,\wr\,
      \left(\Sigma_{2}\,\wr\,\Sigma_{3}\right). 
\]
Once again we note that the isomorphism of 
Lemma~\ref{lemma: transitive gives wreath} is not canonical. It
depends on a choice of a particular class $\cclass\in\class{\lambda_{i}}$ 
and of an isomorphism
\[
\n\cong \cclass\times \class{\lambda_{i}}.
\]
\end{example}
\bigskip

A situation of particular interest comes about when a subgroup 
$D\subseteq\isogroupof{\Lambda}$ acts transitively and {\it freely} 
on $\class{\lambda_i}$. 

\begin{lemma}\label{lemma: free and transitive}
  With the notation of Lemma~\ref{lemma: transitive gives wreath},
  suppose that a subgroup $D\subseteq\isogroupof{\Lambda}$ acts freely
  and transitively on $\class{\lambda_i}$.
  Then the isomorphism \eqref{eq: iso to wreath} identifies $D$ with 
  a transitive subgroup of $K\left(\Lambdabar_{>i}\right)$. If 
  $N=N_{K\left(\Lambdabar\right)}(D)$, then the normalizer of 
  $D$ in $\isogroupof{\Lambda}$ is identified by 
  \eqref{eq: iso to wreath} with the group
\[
\Diag(K(\cclass))\times \Nwiggle.
\]
In particular, $\isogroupof{\Lambda}$ contains a subgroup isomorphic to
$K(\cclass)$ that centralizes $D$.
\end{lemma}

\begin{proof}
  The free and transitive action of $D$ on $\class{\lambda_{i}}$
  allows us to choose a bijection between $\class{\lambda_{i}}$ and
  the underlying set of $D$ such that the given action of $D$ on
  $\class{\lambda_{i}}$ corresponds to the action of $D$ on its
  underlying set by left translation. This choice identifies $D$ with
  a transitive subgroup of $K\left(\Lambdabar\right)$.
The lemma's statement about the normalizer of $D$ is now a consequence of
Lemma~\ref{lemma: wreath normalizer}.
\end{proof}

The situation described in this lemma is illustrated in
Example~\ref{example: centric} below.

We need another result, in a similar spirit, that will be applied in
Section~\ref{section: centralizers and involutions} to subgroups
whose action on some $\class{\lambda_{i}}$ in $\Lambda$ 
is transitive but not free.
Let $D$ be an abelian group acting freely on $\n$, and let $\orbits{\n}{D}$
be the set of orbits of that action. Then the action of $D$ on~$\n$
extends canonically along the diagonal inclusion $D\hookrightarrow
D^{\orbits{\n}{D}}$ to an action of $D^{\orbits{\n}{D}}$ on~$\n$. 
Further, because $D$ acts freely on~$\n$, the action of 
$D^{\orbits{\n}{D}}$ on~$\n$ is faithful, that is, 
$D^{\orbits{\n}{D}}\rightarrow \Sigma_{n}$ is a monomorphism. 
Similarly, if $\orbits{\n}{D}\twoheadrightarrow \orbitsprime$ is a surjective
function of sets, then we obtain an inclusion
$D^{\orbitsprime}\hookrightarrow D^{\orbits{\n}{D}}$ that identifies
$D^{\orbitsprime}$ as a subgroup of $D^{\orbits{\n}{D}}$ via a diagonal
inclusion, giving us 
$D^{\orbitsprime}\subseteq D^{\orbits{\n}{D}}\subseteq\Sigma_{n}$. 
Note that it is necessary for the action of $D$ on~$\n$ to be free 
in order for $D^{\orbits{\n}{D}}$ to include into~$\Sigma_{n}$ as a subgroup.

Continuing to assume that $D$ is an abelian group acting freely on~$\n$, 
suppose that $S\subseteq D$ is a subgroup. Then 
$S^{\orbits{\n}{D}}\subseteq D^{\orbits{\n}{D}}$, so in particular
$S^{\orbits{\n}{D}}$ commutes with~$D\subseteq D^{\orbits{\n}{D}}$. 
Hence if $\orbits{\n}{D}\epi\orbitsprime$ is an epimorphism, 
we can define $D\oplus_{S} S^{\orbitsprime}$ to be the pushout 
in the category of abelian groups of the diagram
$D \leftarrow S \rightarrow S^{\orbitsprime}$. 
Then $D\oplus_{S} S^{\orbitsprime}$ is an abelian group with 
\[
D\subseteq \ D\oplus_{S} S^{\orbitsprime}\ 
 \subseteq \ D^{\orbitsprime}\ 
 \subseteq\ \Sigma_{n}.
\]
Note that if $S$ is a
nontrivial subgroup and $|\orbitsprime|>1$ then 
$D\oplus_{S} S^{\orbitsprime}\subseteq\Sigma_{n}$ 
strictly contains~$D$. In particular, suppose that
$\lambda$ is a partition stabilized by~$D$, and recall that
$\orbit{\lambda}{D}$ denotes the minimal coarsening of $\lambda$ whose
classes are unions of $D$-orbits. (See Section~\ref{section: fixed
  point sets}.)  There is an evident epimorphism
$\orbits{\n}{D}\epi\class{\orbit{\lambda}{D}}$, and we use it to define the
pushout $D\oplus_{S} S^{\class{\orbit{\lambda}{D}}}$.

With these preliminaries in place, we can state our next result.  

\begin{lemma}\label{lemma: transitive not free}
  Let $D\subseteq\Sigma_{n}$ be an abelian group acting freely on~$\n$, 
  let $S$ be a subgroup of~$D$, and let $\lambda$ be a 
  partition of $\n$ that is stabilized by~$D$. 
Then:
\begin{enumerate}
\item If $\mu$ is a refinement of $\lambda$ (or $\lambda$ itself) and
  is stabilized by $D$ then $\mu$ is stabilized by $D\oplus_{S}
  S^{\class{\orbits{\lambda}{D}}}$.
 \item If $\mu$ is a coarsening of $\orbits{\lambda}{S}$ that is
   stabilized by~$D$, then $\mu$ is stabilized by 
   $D\oplus_{S} S^{\class{\orbits{\lambda}{D}}}$. 
\end{enumerate}
\end{lemma}

\begin{proof}
In both cases $\mu$ is assumed to be stabilized by~$D$, so it is 
sufficient to prove that $\mu$ is stabilized 
by~$S^{\class{\orbits{\lambda}{D}}}$ in order to conclude that
$\mu$ is stabilized by $D\oplus_{S} S^{\class{\orbits{\lambda}{D}}}$. 
In fact, because $\orbits{\lambda}{S}\epi \orbits{\lambda}{D}$, we know 
that 
$S^{\class{\orbits{\lambda}{D}}}\subseteq S^{\class{\orbits{\lambda}{S}}}$, 
so it is sufficient to prove that $\mu$ is stabilized 
by~$S^{\class{\orbits{\lambda}{S}}}$. 

  Suppose that $\mu$ is a refinement of $\lambda$ and 
  is stabilized by $D$. 
  If $x\sim_{\mu}y$, then $x\sim_{\lambda}y$ also, and hence
  $x\sim_{\orbits{\lambda}{S}}y$, say
  $x,y\in{\bf{z}}$ where ${\bf{z}}\in\class{\orbits{\lambda}{S}}$.
  But the action of $S^{\class{\orbit{\lambda}{S}}}$ on ${\bf{z}}$
  factors through projection to the factor of S corresponding to~${\bf{z}}$.
  Since $S\subseteq D$ itself stabilizes~$\mu$ by assumption, we know
  that $\sigma x\sim_{\mu}\sigma y$ for any $\sigma\in S$, which
  finishes the proof that $S^{\class{\orbit{\lambda}{S}}}$
  stabilizes~$\mu$.

  Now suppose that $\mu$ is a coarsening of $\orbits{\lambda}{S}$.
  We know that $S^{\class{\orbits{\lambda}{S}}}$ not only 
  stabilizes~$\orbits{\lambda}{S}$, but actually acts trivially on the set of
  equivalence classes of~$\orbits{\lambda}{S}$. Hence 
  $S^{\class{\orbits{\lambda}{S}}}$ stabilizes any coarsening of 
  $\orbits{\lambda}{S}$, and in particular, stabilizes~$\mu$. 
  This finishes what is needed for the second statement of the
  lemma.
\end{proof}

\begin{example}
Consider the following partition
\[
\begin{array}{ccc}
\!\!\lambda:
     & \ \ \{1,2\}\{3,4\}\{5\}\{6\}\ \ 
     &\!\!\{7,8\}\{9,10\}\{11\}\{12\}. 
\end{array}
\]
Let $D\cong\integers/4$ be the subgroup of $\Sigma_{12}$ 
generated by the following product of cycles:
\[
\rho = (1,7, 3, 9)(2,8,4,10)(5,11,6, 12).
\]
Then $D$ acts freely on the set $\{1, \ldots, 12\}$, and $D$ preserves 
$\lambda$. The partition
$\orbit{\lambda}{D}$ is given by merging classes of $\lambda$
that contain elements of the same orbit of~$D$, so we have
\[
\orbit{\lambda}{D}:\ \  
\{1,2,3,4,7,8,9,10\}, \  \{5, 6, 11,12\}.
\]

Let $S\subset D$ be the subgroup isomorphic to $\integers/2$. The group $S$ 
is generated by $\rho^2$,
which is the following product of transpositions
\[
(1,3)(7,9)(2,4)(8,10)(5,6)(11,12).
\]
Accordingly, $D\oplus_S S^{\class{\orbits{\lambda}{D}}}$ is the
subgroup of $\Sigma_n$ generated by $D$ and the elements
$(1,3)(7,9)(2,4)(8,10)$ and $(5,6)(11,12)$.

The partition $\orbit{\lambda}{S}$ is given by merging classes of 
$\lambda$ that contain elements of the same orbit of~$S$, so we have
\[
\orbit{\lambda}{S}:\ \  
\{1,2,3,4\}, \ \{5,6\}, \ \{7,8,9,10\}, \  \{11,12\}.
\]

The reader is invited to check that $D\oplus_S S^{\orbit{\lambda}{D}}$
is an abelian group containing $D$ that preserves $\lambda$ and
$\orbit{\lambda}{S}$, as well as any $D$-invariant refinement of
$\lambda$ or coarsening of $\orbit{\lambda}{S}$. 
\end{example}

\section{Centralizers and involutions}
\label{section: centralizers and involutions}

Fix a prime~$p$. 
In Section~\ref{section: background on approximations}, we showed that
under some hypotheses the Bredon homology of a $G$-space $X$ can be
calculated using the approximation $X_{\pgroupsnone}$ of $X$.  Our
next goal is to further reduce the size of the approximating
collection in the special case $G=\Sigma_{n}$ and $X=\Pcal_{n}$. More
specifically, we would like to use the methods of
Section~\ref{section: pruning} to eliminate, as much as possible,
subgroups in $\pgroupsof{\Sigma_{n}}$ whose fixed points on
$\Pcal_{n}$ are not contractible (the ``problematic'' subgroups).
According to Proposition~\ref{proposition: bad subgroups}, these are
elementary abelian $p$-subgroups of $\Sigma_{n}$ that act freely
on~$\n$. 

In this section, we build on the group theory developed in
Section~\ref{section: isotropy groups} to study the centralizers of
problematic subgroups inside of isotropy groups of~$\Pcal_{n}$.
The main result is 
Proposition~\ref{proposition: group theory cases}, which gives us 
the algebraic data needed to eliminate 
problematic subgroups of $\Sigma_{n}$ that act
non-transitively on~$\n$. Thus we will conclude 
in Section~\ref{section: centralizers} that the only problematic
subgroups that must be included in the approximating collection are
the transitive ones. 

Throughout this section, let $D$ be an abelian $p$-group that acts
freely and non-transitively on~$\n$. (It is not necessary to assume
that $D$ is elementary.)  Let 
$\Lambda=\left(\lambda_{0}<\lambda_{1}<...<\lambda_{j}\right)$ 
be a simplex of $\Pcal_n$, and let
$K=\isogroupof{\Lambda}\subseteq\Sigma_{n}$ be the isotropy group of
$\Lambda$. Assume that $D\subseteq K$ (so $D$ stabilizes $\Lambda$).
As usual, let $C_{K}(D)$, $N_{K}(D)$, and $W_{K}(D)=N_{K}(D)/D$ denote the
centralizer, normalizer, and Weyl group of $D$ in~$K$, respectively.
Let $\coeff$ be a coefficient system for $\Sigma_n$ that takes values
in $\integers_{(p)}$-modules. By Proposition~\ref{proposition: pruning
  criterion}, in order to eliminate $D$ we need to show that the
collection $\pgroupsntof{W_{K}(D)}$ is $\coeff(\Sigma_n/D)$-ample.

The usual strategy is to use 
Proposition~\ref{proposition: algebraic pruning criterion}, for which
we need to show that $W_{K}(D)$ has an element of order $p$ that acts
trivially on $\coeff(\Sigma_n/D)$.  Typically such elements are found
in $C_{K}(D)/D$, and therefore we would like to know that $C_{K}(D)/D$
has elements of order $p$.  Subgroups that do not satisfy this
condition are called \defining{$p$-centric}. It turns out that in
``most'' cases $D$ is not $p$-centric in~$K$ (so it
can be eliminated using 
Proposition~\ref{proposition: algebraic pruning criterion}). This is
the first case of Proposition~\ref{proposition: group theory cases}.
However, in some important cases $D$ is $p$-centric
in~$K$, as in the following example.

\begin{example}\label{example: centric}
  Let $p=3$ and $n=18$.  Recall that a transitive elementary abelian
  $3$-group $\Delta_{2}\cong(\integers/3)^{2}$ of $\Sigma_{9}$ is
  given by the action of $(\integers/3)^{2}$ on its own elements by
  translation.  Let $\Lambda$ consist of a single
  partition~$\lambda_{0}$, which partitions $\n$ by the nine
  two-element sets $\{2i-1,2i\}$. Let $D\cong (\integers/3)^{2}$ be
  the diagonal embedding of $\Delta_{2}$ in $\Sigma_{18}$ as
  permutations of odd integers and even integers. That is, $D$ acts
  transtively on $\{1,3,\dots,17\}$ and on $\{2,4,\dots,18\}$ and
  satisfies $d(2i-1)+1=d(2i)$, so $D$ stabilizes $\lambda_{0}$ and
  acts transitively on its classes.

The isotropy group
$K=\isogroupof{\Lambda}$ is $\Sigma_2\wr\Sigma_9\subseteq\Sigma_{18}$, and
$D=\tilde{\Delta}_{2}$ (see Definition~\ref{defn: special groups}).
By Lemma \ref{lemma: wreath normalizer}, the normalizer of $D$ in $K$ is
  $N_{\isogroupof{}}(D)
          =\Sigma_2\times \Nwiggle$, where 
$N=N_{\Sigma_9}\left(\Delta_{2}\right)$, and by inspection the centralizer
is $C_{\isogroupof{}}(D) =\Sigma_2\times \Cwiggle$, where 
$C=C_{\Sigma_9}(\Delta_{2})$. 
However, $\Delta_{2}$ is self-centralizing in~$\Sigma_{9}$. Hence
$C_{\isogroupof{}}(D)/D\cong\Sigma_{2}$ and has no elements
of order~$3$, and $D$ is $3$-centric in~$\isogroupof{}$.
 
Even worse, the Weyl group of $D$ in $\isogroupof{}$ is
$W_{\isogroupof{}}(D) =\Sigma_2 \times \GL_2(\field_3)$.  This
means that $D$ is also ``$3$-radical'' in~$\isogroupof{}$,
i.e., the Weyl group has no nontrivial normal $3$-subgroups. As
a consequence, it is difficult to eliminate $D$ by standard methods.
\end{example}

We will show in Proposition~\ref{proposition: group theory cases} that
$p$-radical situations like Example~\ref{example: centric}
can occur only when the prime $p$ is odd, and that in this case
the Weyl group $W_{K}(D)$ has an odd involution that acts trivially on
the poset~$\pgroupsntof{W_{K}(D)}$. 
(In Example~\ref{example: centric}, the
involution comes from the factor $\Sigma_2$ of $\Sigma_2\times
\GL_2(\field_3)$.)  By assumption~\eqref{condition: involution} in
Theorem~\ref{theorem: main theorem}, such an involution will act by
$-1$ on $\coeff(\Sigma_n/D)$. It will turn out that the presence of
this involution implies that $\pgroupsntof{W}$ is ample for the
trivial reason that all the relevant homology and cohomology groups
vanish. (We learned from the referee that this type of argument is
sometimes referred to as ``center kills.'')

The following proposition is the main result of this section. It says
that either $D$ is not $p$-centric in $K$, or the centralizer of $D$
has an odd involution that acts trivially on the poset
$\pgroupsntof{W}$. The proof of the proposition occupies the remainder
of this section. By an ``odd involution'' we mean a permutation of
order~$2$ that can be written as a product of an odd number of
transpositions.

\begin{proposition} \label{proposition: group theory cases} 
  Let $K\in\Iso\left(\Pcal_{n}\right)\cup\{\Sigma_{n}\}$, and let
  $D\subseteq K$ be an abelian $p$-subgroup of $\Sigma_{n}$ that acts
  freely and non-transitively on $\n$.  Let $C_{K}(D)$ be the centralizer of
  $D$ in $K$. Then either
\begin{enumerate}
\item $p\mid [C_{K}(D):D]$, or 
\item $p$ is odd, and there is an odd involution in $C_{K}(D)$ that acts 
trivially on the poset of $p$-subgroups of the normalizer of $D$ in~$K$.  
\end{enumerate}
\end{proposition}

\begin{proof}
Recall the following construction from the end of
Section~\ref{section: isotropy groups}. Suppose that 
$S$ is a subgroup of $D\subseteq\Sigma_{n}$ 
and $\lambda$ is a partition of~$\n$. We define
$D\oplus_{S} S^{\class{\orbits{\lambda}{D}}}$ as the pushout in the 
category of abelian groups of the diagram 
$D\leftarrow S\rightarrow S^{\class{\orbits{\lambda}{D}}}$. This 
pushout is an abelian subgroup of~$\Sigma_{n}$ that contains~$D$.
By Lemma~\ref{lemma: transitive not free}, 
$D\oplus_{S} S^{\class{\orbits{\lambda}{D}}}$ stabilizes
any $D$-invariant refinement of $\lambda$ and any $D$-invariant 
coarsening of~$\orbits{\lambda}{S}$. 

To prove the proposition, first
suppose that $K=\Sigma_{n}$, and let $\mu$ be the partition of $\n$
by the orbits of~$D$. Because $D$ does not act transitively on~$\n$, the 
partition $\mu$ has more than one equivalence class. 
Then $C_{K}(D)$ contains the 
subgroup~$D^{\class{\mu}}\cong D\oplus_{D}D^{\class{\orbits{\mu}{D}}}$, which 
is a $p$-subgroup of $\Sigma_{n}$ strictly containing~$D$. 
Hence $p\mid [C_{K}(D):D]$.

Now suppose $K\in\Iso\left(\Pcal_{n}\right)$ is the isotropy group of 
a nondegenerate chain of proper, nontrivial partitions, say
\[
\Lambda= \left(\lambda_{0}<\lambda_{1}<...<\lambda_{j}\right). 
\] 
By assumption, $D\subseteq K$ so $D$ stabilizes $\Lambda$.
For each $i$, the group $K$ acts on $\class{\lambda_{i}}$, 
and therefore so does $D$. Let $i$ be the
smallest number for which the action of $D$ on the set of classes of
$\lambda_i$ is transitive. (If $D$ does not act transitively even
on $\class{\lambda_{j}}$, then we interpret the argument that follows
with $i=j+1$, and we understand $\lambda_{j+1}$ to be the indiscrete
partition of~$\n$, consisting of just one equivalence class. Likewise, if
$i=0$, we use the convention that $\lambda_{-1}$ is the discrete 
partition of $\n$ into singleton sets.) 

Let $S\subset D$ be the subgroup that acts
trivially on the set of classes of $\lambda_i$, i.e.,
\[
S\definedas\ker\left(D\rightarrow\Sigma_{\class{\lambda_{i}}}\right). 
\]
There are two cases: \emph{(i)} $S$ is nontrivial, in which case it
turns out that $p\mid [C_{K}(D):D]$; or \emph{(ii)} $S=\{e\}$, in
which case it turns out that sometimes $p\mid [C_{K}(D):D]$, but if 
not, then $p$ is odd and there is an odd
involution in $C_{K}(D)$ that acts trivially on the poset of
$p$-subgroups of the normalizer of $D$ in~$K$.

Suppose first that $S$ is a nontrivial group.  Let 
$\mu=\orbits{\lambda_{i-1}}{D}$, that is, $\mu$ is the
finest mutual coarsening of the partition $\lambda_{i-1}$ and the
partition of $\n$ by the orbits of~$D$. 
Consider the group $D\oplus_S S^{\class{\mu}}$.
Since $D$ acts non-transitively on~$\class{\lambda_{i-1}}$ by
assumption, we know that $\mu$ has more than one equivalence
class, so $D\oplus_S S^{\class{\mu}}$ is an abelian $p$-group
that strictly contains~$D$. 
Lemma~\ref{lemma: transitive not free} implies immediately that
$D\oplus_S S^{\class{\mu}}$ stabilizes
$\lambda_{0}<\dots<\lambda_{i-1}$. 
Further, by construction $S$ acts trivially on~$\class{\lambda_{i}}$,
that is, the classes of $\lambda_{i}$ are unions of orbits of~$S$. 
Since $\lambda_{i}$ is also a coarsening of~$\lambda_{i-1}$, this 
means that $\lambda_{i}$ is a $D$-invariant coarsening 
of~$\orbits{\lambda_{i-1}}{S}$. Lemma~\ref{lemma: transitive not free} 
now tells us that $\lambda_{i}$ (as well as $\lambda_{i+1}$,\dots,
$\lambda_{j}$) is stabilized by~$D\oplus_S S^{\class{\mu}}$. That is,
$D\oplus_S S^{\class{\mu}}$ stabilizes all of~$\Lambda$, so 
$D\oplus_S S^{\class{\mu}}\subset K$. 
Therefore $p\mid [C_{K}(D):D]$ in this case.

Now suppose that $S$ is trivial. Then $D$ acts freely and transitively
on the classes of $\lambda_i$, and we are in the situation 
of Lemma~\ref{lemma: free and transitive}. Let $\cclass$ be a fixed class of
$\lambda_i$, and let $K(\cclass)$ be the isotropy group of the restriction
of $\Lambda$ to $\cclass$. Let 
$G=\im\left(K\hookrightarrow\Sigma_{\class{\lambda_{i}}}\right)$. 
Notice that since $D$ acts freely on $\class{\lambda_{i}}$, we 
can regard $D$ as a subgroup of~$G$. 
By Lemma~\ref{lemma: free and transitive}, 
$K\cong K(\cclass)\,\wr\,G$. In particular, $K$ 
contains a subgroup $\Diag(K(\cclass))$, isomorphic to $K(\cclass)$, that
centralizes $D$ and has trivial intersection with $D$. Thus $C_K(D)/D$
contains a subgroup isomorphic to $K(\cclass)$. If $p$ divides the order of
$K(\cclass)$ then $p\mid [C_{K}(D) : D]$, and we are done.

Suppose that $p$ does not divide the order of $K(\cclass)$.  Since $K$ acts
transitively on the classes of $\lambda_i$, by 
Lemma~\ref{lemma: transitive} the restriction of $\lambda_0$ to any
class of $\lambda_i$ is isomorphic to the restriction of $\lambda_0$
to~$\cclass$. By assumption, the partition $\lambda_0$ is not discrete, and 
therefore the restriction of $\lambda_0$ to $\cclass$ is not discrete.

It follows that there exist elements $x,y\in \cclass$ that
belong to the same class of $\lambda_0$. The transposition $(x\ y)$
that interchanges these two elements is an element of $K(\cclass)$. In
particular, the order of $K(\cclass)$ is always divisible by $2$, so
if $p\nmid |K(\cclass)|$, then $p>2$. Now let $\sigma$ be the image of
$(x\ y)$ under the diagonal embedding $K(\cclass)\hookrightarrow
K\cong K(\cclass)\,\wr\,G$.  The element $\sigma$ is a product of
disjoint transpositions, as many transpositions as there are elements
of $D$, which is a power of~$p$.  In particular the number of transpositions
is odd, so $\sigma$ is an odd involution.  The element $\sigma$ is in
$\Diag(K(\cclass))$ and therefore it centralizes $D$.

Moreover, by Lemma~\ref{lemma: wreath normalizer}, we have
$N_{K}(D)\cong K(\cclass)\times N_G(D)$. 
Since $p$ does not divide the order of $K(\cclass)$, every $p$-subgroup of
$K(\cclass)\times N_G(D)$ is contained in $N_G(D)$, and is centralized by
$K(\cclass)$. It follows that $\sigma$ centralizes every $p$-subgroup 
of~$N_{K}(D)$, and in particular
$\sigma$ acts trivially on the poset of such $p$-subgroups.
\end{proof}

\section{Eliminating Problematic Subgroups}
\label{section: centralizers}

We saw in Section~\ref{section: fixed point sets} that if $D$ is a
problematic $p$-subgroup of $\Sigma_{n}$ (i.e.,
$\left(\Pcal_{n}\right)^D$ is not contractible), then $D$ is an
elementary abelian $p$-group that acts freely on $\n$. In this
section, we introduce two conditions on a Mackey functor $\coeff$ for
$\Sigma_n$. In the main result of this section, 
Proposition~\ref{proposition: pruning proposition}, we show that
these conditions permit the use of
Proposition~\ref{proposition: pruning criterion} to discard the
problematic subgroups when they do not act transitively on $\n$. 
The conditions hold for the Mackey functors that we have in mind for 
applications. (See Section~\ref{section: our coefficients}.)

The $\Sigma_{n}$-spaces we need to approximate using the methods of
previous sections are $\Pcal_{n}$ and~$\ast$.  Hence throughout this
section we assume an ambient isotropy subgroup
$K\in\Iso\left(\Pcal_{n}\right)\cup\{\Sigma_{n}\}$.  We must show that
for any $K$ containing a problematic subgroup $D$, the ampleness
condition of Proposition~\ref{proposition: pruning criterion}
is met. To apply the methods of 
Section~\ref{section: pruning}, we need centralizing
elements that act trivially on coefficients. We need an appropriate 
characterization of the coefficients that will apply in our cases of
interest (Section~\ref{section: our coefficients})
and be sufficient to guarantee the existence of the needed elements.

Recall from Example~\ref{example: relevant-Mackey-functor} that our
prime examples of Mackey functors for applications are constructed by
means of certain homotopy functors applied to $S^n$.  Note that a
permutation $\sigma\in \Cen_{\Sigma_{n}}(D)$ induces a $D$-equivariant
map $\sigma_\sharp\colon S^n\to S^n$.  It turns out to be useful to
look at what happens when we pass to the general linear group by
embedding $\Sigma_{n}\hookrightarrow \GL_{n}\reals$ as permutations of
the standard basis.  Let $\Cen_{\GL_n\reals}(D)$ denote the
centralizer of $D$ in $\GL_{n}\reals$.  Notice that if
$\sigma\in\ker\left[ \Cen_{\Sigma_{n}}(D)\rightarrow
  \pi_{0}\Cen_{\GL_n\reals}(D)\right]$, then the map
$\sigma_\sharp\colon S^n\to S^n$ is actually $D$-equivariantly
homotopic to the identity map, which is a good sign for trivial action
on our coefficients.

\begin{definition}\label{definition: centralizer condition}
  We say that $\coeff$ \defining{satisfies the centralizer condition for
    $D$} if the kernel of $\Cen_{\Sigma_{n}}(D)
  \rightarrow\pi_{0}\Cen_{\GL_n\reals}(D)$ acts trivially on
  $\coeff\left(\Sigma_{n}/D\right)$.
\end{definition}

As discussed in Section~\ref{section: our coefficients}, our primary
examples of Mackey functors satisfy the centralizer condition.  The
centralizer condition allows us to use 
Proposition~\ref{proposition: algebraic pruning criterion} to
eliminate problematic groups that are not $p$-centric in the ambient
isotropy group~$K$. 

In the following definition, ``odd involution'' means an odd permutation 
of order $2$. 

\begin{definition}     \label{definition: involution condition}
  Let $p$ be an odd prime.  We say that $\coeff$ \defining{satisfies
  the involution condition for $D$} if any odd involution 
  in $\Cen_{\Sigma_{n}}(D)$ acts on $\coeff\left(\Sigma_{n}/D\right)$
  by multiplication by $-1$. 
\end{definition}
The involution condition enables us to eliminate problematic subgroups
in the few cases when they happen to be $p$-centric. 

\begin{definition} \label{definition: Mackey functor condition} 
  We say that the Mackey functor $\coeff$ \defining{satisfies the
    centralizer condition} (resp. \defining{satisfies the involution
    condition}) if it satisfies the corresponding condition in
  Definition~\ref{definition: centralizer condition}
  (resp. Definition~\ref{definition: involution condition}) for all
  elementary abelian $p$-subgroups of $\Sigma_{n}$ that act freely and
  non-transitively on $\n$.
\end{definition}
Examples of Mackey functors that satisfy both the centralizer and the
involution conditions are given in 
Section~\ref{section: our coefficients}.

Our main result in this section is the following proposition. 
The term ``ample'' was
defined at the beginning of Section~\ref{section: pruning}. 

\begin{proposition} \label{proposition: pruning proposition}
Let $D\subset\Sigma_{n}$ be an elementary abelian $p$-subgroup
that acts freely and non-transitively on $\n$, and let $\coeff$ be a
Mackey functor for $\Sigma_n$ taking values in $\integers_{(p)}$-modules. 
Assume that 
\begin{itemize}
\item $\coeff$ satisfies the centralizer condition for $D$, and
\item if $p$ is odd, $\coeff$ satisfies the involution condition for $D$. 
\end{itemize}
Then for any $K\in\Iso\left(\Pcal_{n}\right)\cup\{\Sigma_{n}\}$ 
such that $D\subseteq K$, we have that
$\pgroupsntof{W_K(D)}$ is $\coeff\left(\Sigma_n/D\right)$-ample. 
\end{proposition}

Before the proof, we need a lemma from representation theory, which 
follows immediately from~\cite[Thm 1.3.4]{Benson}. 

\begin{lemma}
\label{lemma: centralizers in GL}
If $D\subseteq\GL_n\reals$ is finite, then
$\pi_0\Cen_{\GL_n\reals}(D)$ is an elementary abelian $2$-group.
\end{lemma}
% DO NOT DELETE
% DO NOT DELETE
%\begin{proof}
% If $D$ is elementary abelian, then the quaternion case doesn't occur,
% although the relevant GL_i is path-connected anyway. 
%
% If $D$ is an elementary abelian $p$-group, then we have the following.
%  If $p=2$, then the
%   indicated centralizer is a product of copies of $\GL_{a_i}\reals$
%   for various $a_i$. If $p$ is odd, the centralizer is a product of
%   $\GL_a\reals$ with a product of copies of $\GL_{b_i}\complexes$ for
%   various $a$, $bn_i$. 
%\end{proof}
% DO NOT DELETE
% DO NOT DELETE

We have everything we need for the odd-primary case, so we handle
this first. 

\begin{proof}
[Proof of Proposition~\ref{proposition: pruning proposition} 
        for $p$ odd, $p$ dividing $|C_{K}(D)/D|$]
\mbox{}\\
 Pick $x\in \Cen_{K}(D)/D$ of
 order~$p$. If $\xwiggle\in C_{K}(D)$ is an inverse image of $x$, it is
 clear from Lemma~\ref{lemma: centralizers in GL} that $\xwiggle$ belongs to
 the kernel of $C_{K}(D)\to\pi_0\Cen_{\GL_n\reals}(D)$. In view of
 the centralizer condition, $x$ acts trivially on
 $\coeff(\Sigma_n/D)$. The desired conclusion follows from 
 Proposition~\ref{proposition: algebraic pruning criterion}.
\end{proof}

\begin{proof}
[Proof of Proposition~\ref{proposition: pruning proposition} 
       for $p$ odd,  $p$ not dividing $|C_{K}(D)/D|$]
\mbox{}\\
We will show that 
$\pgroupsntof{W_{K}(D)}$ is $M$-ample with $M=\coeff(\Sigma_n/D)$ by 
showing that all of the relevant homology and cohomology groups vanish.
To declutter the notation for quotients, let $\Cbar=C_{K}(D)/D$, let 
$\Nbar=N_{K}(D)/C_{K}(D)$, and let $W=W_{K}(D)$. 
The short exact sequence
\begin{equation}\label{equation:ses}
     1\to \Cbar \to W \to \Nbar \to 1  
\end{equation}
shows that the map $\lvert\pgroupsntof{W}\rvert_{hW} \to BW$ can be written as 
\[
\left(\lvert\pgroupsntof{W}\rvert_{h\Cbar}\right)_{h\Nbar}\rightarrow \left(B\Cbar\right)_{h\Nbar}. 
\]
The Serre spectral sequence shows that for the homology case, 
it is enough to prove that the local coefficient
groups 
$H_*(\lvert\pgroupsntof{W}\rvert_{h\Cbar};M)$ and 
$H_*(B\Cbar;M)$ vanish
(respectively, for cohomology,
$H^*(\lvert\pgroupsntof{W}\rvert_{h\Cbar};M)$ and
$H^*(B\Cbar;M)$ vanish).

We will handle the case $H_*(\lvert\pgroupsntof{W}\rvert_{h\Cbar};M)$;
the others are similar.  
By Proposition~\ref{proposition: group theory cases}, there exists an
odd involution $\tau\in C_{K}(D)$ that acts trivially on 
poset of $p$-subgroups
of $N_{K}(D)$, and $\tau$ projects to an involution $\taubar\in \Cbar$. The
element $\taubar$ acts trivially on the
space~$\lvert\pgroupsntof{W}\rvert$ and, in view of the involution
condition, acts by $-1$ on $\coeff(\Sigma_n/D)$.  Consider the Serre
spectral sequence of
\[
   \lvert\pgroupsntof{W}\rvert \to \lvert\pgroupsntof{W}\rvert_{h\Cbar} \to B\Cbar\,.
\]
Since $M$ is a $\integers_{(p)}$-module
 and $\Cbar$ has order prime to $p$, we know
that $E^{2}_{i,j}=0$ for $i>0$, while the group $E^2_{0,j}$ is given
by the coinvariants of the action of $\Cbar$ on
$H_{j}(\lvert\pgroupsntof{W}\rvert;M)$. However, $\taubar\in\Cbar$ acts trivially
on $\lvert\pgroupsntof{W}\rvert$ and acts on $M$ by $-1$, so $\taubar$ acts on
$H_{j}(\lvert\pgroupsntof{W}\rvert;M)$ by $-1$. Since the coinvariants of this
$\taubar$-action vanish, the groups 
$H_{0}\left(\Cbar; H_{j}(\lvert\pgroupsntof{W}\rvert;M) \right)$ vanish for all
$j$, and the Serre spectral sequence collapses to zero at $E^{2}$.
\end{proof}

The remainder of the proof of 
Proposition~\ref{proposition: pruning proposition}, for $p=2$,
requires two known lemmas. As usual, $\lvert\pgroupsntof{G}\rvert$ denotes the
nerve of the poset of nontrivial $p$-subgroups of a finite group $G$,
and the group $G$ acts on $\pgroupsntof{G}$ by conjugation. The
following lemma is due to Grodal.

\begin{lemma}[\cite{Grodal}, Proposition 5.7] \label{lemma: normal subgroup}
  Let $G$ be a finite group with a normal subgroup $H$ of order prime
  to~$p$.  Then $ \lvert\pgroupsntof{G}\rvert/H$ is isomorphic 
  to~$\lvert\pgroupsntof{G/H}\rvert$.
\end{lemma}

The following well-known lemma is due to Quillen. See the proof of
Proposition~2.4 in~\cite{Quillen}. 
\begin{lemma}[\cite{Quillen}]\label{lem:normal-implies-contractible}
  If $G$ is a finite group with a nontrivial normal
  $p$-subgroup, then $\lvert\pgroupsntof{G}\rvert$ is contractible.
\end{lemma}

Let $\myker$ denote the subgroup of $C_{K}(D)$ generated by $D$ and the
kernel of the map $C_{K}(D)\to\pi_0\GL_n\reals$. It is clear from
Lemma~\ref{lemma: centralizers in GL} that $C_{K}(D)/\myker$ is an elementary
abelian $2$-group.

\begin{proof}[Proof of Proposition~\ref{proposition: pruning proposition} 
              for $p=2$, $p$ dividing $|\myker/D|$]
This follows by the same reasoning as in the case $p$ odd, 
$p$ dividing $|C_{K}(D)/D|$.
\end{proof}

\begin{proof}[Proof of Proposition~\ref{proposition: pruning proposition} 
              for $p=2$, $p$ not dividing $|\myker/D|$]

By Proposition~\ref{proposition: group theory cases}, $p$ divides 
$|C_{K}(D)/D|$, so
under the assumption $p\nmid |\myker/D|$, it must be the case that
$p\mid|C_{K}(D)/\myker|$ and $C_{K}(D)/\myker$ 
is nontrivial. Observe that $\myker$ is a normal
subgroup of $N_{K}(D)$, because it is generated by the normal subgroup $D$
and the normal subgroup obtained by intersecting $C_{K}(D)$ with the kernel
of $N_{K}(D)\to\pi_0\GL_n\reals$. It follows that  $C_{K}(D)/\myker$ is a
nontrivial normal $2$-subgroup of $N_{K}(D)/C_{0}$. 

Let $W=W_{K}(D)$, and as usual, let $\lvert\pgroupsntof{W}\rvert$
denote the nerve of the poset of nontrivial $p$-subgroups of $W$. We
must show that $\lvert\pgroupsntof{W}\rvert_{hW}\to BW$ induces
isomorphisms on twisted $M$-homology and $M$-cohomology, where
$M=\coeff\left(\Sigma_{n}/D\right)$.  Let $\mykerbar$ denote
$\myker/D$. As in \eqref{equation:ses}, we have a short exact sequence
\[
   1 \to \mykerbar \to W \to N/\myker\to 1, 
\]
and a Serre spectral sequence argument like that following
\eqref{equation:ses} establishes that we need only show that the map
$\lvert\pgroupsntof{W}\rvert_{h\mykerbar}\to B\mykerbar$ induces an
isomorphism on $M$-homology and $M$-cohomology. Further, because of
the definition of $\myker$ and the assumption that $\coeff$ satisfies
the centralizer condition, the action of $\mykerbar$ on $M$ is
trivial, so we have untwisted coefficients.  Consider the following
commutative diagram comparing homotopy orbits to strict orbits:
\[
\begin{CD}
\lvert\pgroupsntof{W}\rvert_{h\mykerbar}@>>>(*)_{h\mykerbar}\\
@VVV @VVV\\
\lvert\pgroupsntof{W}\rvert/\,\mykerbar@>>>*
\end{CD}  \,.
\]
Since $M$ is a $\integers_{(p)}$-module, it is enough to show that all
of the maps in this diagram are $p$-local equivalences.  By
Lemma~\ref{lemma: normal subgroup}, the orbit space
$\lvert\pgroupsntof{W}\rvert/\mykerbar$ is isomorphic to the nerve of
the poset of nontrivial $p$-subgroups of $N_{K}(C)/C_{0}$, and that poset is
weakly contractible (Lemma~\ref{lem:normal-implies-contractible}),
because $N_{K}(C)/C_{0}$ has the nontrivial normal $p$-subgroup
$C_{K}(D)/\myker$. Thus the lower horizontal map is an equivalence.  The
right vertical map is a $p$-local equivalence because the order of
$\mykerbar$ is prime to~$p$.  In the same way the isotropy groups of
the action of $\mykerbar$ on $\lvert\pgroupsntof{W}\rvert$ are all of
order prime to $p$, and so the isotropy spectral sequence of this
action \cite[2.4]{Sharp} shows that the left vertical map is a
$p$-local equivalence as well.
\end{proof}

\begin{remark}
Grodal has pointed out to us that \cite{Grodal} can be used for an 
alternative proof of Proposition~\ref{proposition: pruning proposition}. 
It is shown in \cite[Ex. 8.6 and \S 9]{Grodal} that 
$\pgroupsntof{W}$ is cohomologically $M$-ample if and only if 
$\Hom\left(\Stein_*(W), M\right)$ is acyclic, where $\Stein_{*}(W)$
is the Steinberg complex of~$W$ as defined in~\cite{Grodal}. Dually, 
$\pgroupsntof{W}$ is homologically $M$-ample if and only if
$\Stein_*(W)\otimes M$ is acyclic. To show that the assumptions 
of Proposition~\ref{proposition: pruning proposition} imply 
ampleness, one can use Proposition~\ref{proposition: group theory cases} 
together with the properties of the Steinberg complex from 
\cite[\S 5]{Grodal} to establish acyclicity. This essentially 
representation-theoretic approach could be useful in
other applications and generalizations, where the assumptions on the 
Mackey functor might vary. 
\end{remark}

\section{Results of approximating}
\label{section:approximation-results}

In this section, we assemble the results from previous sections to
establish the main results announced in the introduction and
restated below. 
\begin{maintheorem}   
\maintheoremtext
\end{maintheorem}

Recall that in the following corollary, $\Stein_k$ denotes
$\Hwiggle_{k-1}(B_k^{\diamond};\integers)$ and $R$ denotes the ring
$\integers[\GL_k(\field_p)]$.

\begin{onlyonedimensiontheorem}
\onlyonedimensiontext
\end{onlyonedimensiontheorem}

First we need a small lemma. Let $G$ be a finite group, and let
$\G$ denote the underlying set of $G$. The group $G$ acts on
$\G$ by left translation, and hence on the poset $\Pcal(\G)$ of
nontrivial, proper partitions of~$\G$.

\begin{lemma}\label{lemma:partition-of-free-action} 
The fixed point poset $\Pcal(\G)^G$ is canonically
  isomorphic to the poset of proper, nontrivial subgroups
  of~$G$.
\end{lemma}

\begin{proof}
  Let $\lambda$ be a partition of $G$ that is invariant under the
  action of~$G$. To associate a subgroup to the partition, 
  let $G(\lambda)$ be the equivalence class of the
  identity element $e\in G$. We claim that $G(\lambda)$ is a subgroup
  of~$G$. Indeed, let $g_1, g_2\in G(\lambda)$. Since 
  $e\sim_\lambda g_1$ and $\lambda$ is $G$-invariant, it follows that 
  $g_{2}=g_{2}e \sim_\lambda g_{2}g_{1}$. But $g_{2}\sim_\lambda e$ also, 
  and so $g_{2}g_{1}\sim_{\lambda}g_{2}\sim_{\lambda}e$.  
    Thus $G(\lambda)$ is a subgroup of~$G$.  If $\lambda$ is neither the
  discrete nor the indiscrete partition of~$\n$, then $G(\lambda)$ is
  a proper, nontrivial subgroup of~$G$.

  Conversely, to associate a partition of $\G$ to a subgroup~$H\subseteq G$, 
  we take the partition of the set $\G$ by the cosets of~$H$. 

  It remains to check that $\lambda \mapsto G(\lambda)$ and
  $H\mapsto\{gH\}_{g\in G}$ are inverses.  To see this, observe that
\begin{align*}
g_1\sim_\lambda g_2 &\iff g_2^{-1} g_1 \sim_\lambda e\\
                      &\iff g_2^{-1}g_1\in G(\lambda)\\
                      &\iff \mbox{$g_{1}$ and $g_{2}$ are in the
                             same coset of $G(\lambda)$}. 
\end{align*}
\end{proof}

\begin{proof}[Proof of Theorem~\ref{theorem: main theorem}]
  Recall that $\pgroupsof{\Sigma_n}$ denotes the family of all
$p$-subgroups of $\Sigma_{n}$.  Let $\Ccal$ consist of all
$p$-subgroups of $\Sigma_{n}$ except the elementary abelian
$p$-subgroups that act freely on $\n$, and $\Dcal$ the collection
containing all elementary abelian $p$-subgroups of $\Sigma_{n}$ that
act freely and transitively on $\n$.  The collection $\Dcal$ is empty
unless $n=p^k$, in which case it consists entirely of conjugates of
the subgroup $\Delta_k$.  Certainly $\Ccal$ is closed under passage to
$p$-supergroups, and $\Dcal$ is initial in $\Ccal\cup\Dcal$.  Let
$X=\Pcal_n$, and consider the commutative diagram of $\Sigma_n$-spaces
\begin{equation}  \label{equation: fix up Pn}
\begin{CD}
X_{\Ccal\cup\Dcal}@>>>X_{\pgroupsnone} @>>> X \\\
  @VVV  @VVV   @VVV \\
(*)_{\Ccal\cup\Dcal} @>>> (*)_{\pgroupsnone} @>>> *
\end{CD}\,.
\end{equation}
By Proposition~\ref{proposition: family}, the horizontal arrows
in the right-hand square induce isomorphisms on Bredon homology and cohomology with coefficients in $\coeff$ because
$\coeff$ is projective relative to $\pgroupsof{\Sigma_n}$. We assume that  
$\coeff$ takes values in $\integers_{(p)}$-modules, so
Propositions~\ref{proposition: pruning criterion}
and~\ref{proposition: pruning proposition} imply that the horizontal arrows 
in the left-hand square induce isomorphisms in Bredon homology and cohomology. 

If $n\ne p^k$, then $\Dcal$ is empty. Meanwhile,
Proposition~\ref{proposition: bad subgroups} tells us that the left
vertical map is an equivalence on fixed point sets of subgroups in 
$\Ccal$, hence a $\Sigma_{n}$-equivalence. It is
therefore an isomorphism on Bredon homology and cohomology.
Connecting the isomorphisms around the outside of the diagram gives 
Theorem~\ref{theorem: main theorem} for $n\neq p^k$.

Suppose $n=p^k$, so that $\Dcal$ consists of conjugates of
$\Delta_{k}$. In this case, the leftmost vertical arrow in 
\eqref{equation: fix up Pn} is not a Bredon (co)homology isomorphism,
but we can still calculate it.
Lemma~\ref{lemma:fix-up-d} gives a homotopy pushout diagram 
\begin{equation}  \label{diag: htpy pushout for X}
\begin{CD} 
   X_{\Dcal} @>>> X_{\Ccal\cup\Dcal} \\
   @VVV             @VVV\\
   (*)_{\Dcal} @>>> (*)_{\Ccal\cup\Dcal}
\end{CD}\,.
\end{equation}
We can use the explicit formula of
Lemma~\ref{lemma:one-subgroup-approximation} to give a formula for
$X_{\Dcal}$ once we know $X^{\Delta_{k}}$.  As a $\Delta_{k}$-set,
$\n$ is isomorphic to $\Delta_{k}$ acting on itself by left
translation, so by Lemma~\ref{lemma:partition-of-free-action}, we find
that $X^{\Delta_{k}}\cong B_{k}$. Sticking 
\eqref{diag: htpy pushout for X} together with 
\eqref{equation: fix up Pn} and applying
Lemma~\ref{lemma:one-subgroup-approximation} to compute $X_{\Dcal}$
and $(*)_{\Dcal}$ gives us the diagram
\[
\begin{CD} 
   \Sigma_{n}\times_{\Aff_{k}}
       \left(E\GL_{k}\times B_{k}\right)
       @>>>   X_{\Ccal\cup\Dcal} @>>> X \\
      @VVV                          @VVV @VVV\\
     \Sigma_{n}\times_{\Aff_k} \left(E\GL_{k}\right)
        @>>> (*)_{\Ccal\cup\Dcal} @>>> *
\end{CD}\,.
\]
Here the left square is a homotopy pushout and the horizontal maps
are isomorphisms in Bredon (co)homology.  Taking cofibers vertically
gives us the result for~$n=p^k$.
\end{proof}

\begin{proof}[Proof of Corollary~\ref{corollary:only-one-dimension}]
  Let $G=\GL_k$. Since $M(\Sigma_n/\Delta_k)$ is an
  $\integers_{(p)}$-module, we may take $R=\integers_{(p)}[G]$ and
  $\Stein_k=\Hwiggle_{k-1}(B_k^{\diamond};\integers_{(p)})$.

  Let $Y={\Sigma_n}_{+} \wedge_{\Aff_k}\left(EG_{+}\wedge
    B_k^{\diamond}\right)$.  All of the isotropy subgroups of $Y$ are
  conjugate to $\Delta_{k}$, so the calculation of
  Lemma~\ref{lemma:one-subgroup-approximation} tells us that there are
  isomorphisms
\[
\begin{aligned}
 \BredonHomRed{*}{G}{Y}{\coeff}
       &\cong \Hwiggle_*((B_k^\diamond)_{\hobased G}; M)\\
 \BredonCohRed{*}{G}{Y}{\coeff}
       &\cong \Hwiggle^*((B_k^\diamond)_{\hobased G}; M)\,.
\end{aligned}
\]
Consider the local coefficient homology Serre spectral sequence
\[
  E^2_{i,j} =\Tor_i^{R}(\tilde H_j(B_k^\diamond;\integers_{(p)}), M)\Rightarrow
  \Hwiggle_*((B_k^\diamond)_{\hobased G}; M)\,.
\]
The spectral sequence collapses at $E^2$ to give the desired homology
calculation, because $\Hwiggle_j(B_k^\diamond;\integers_{(p)})$
vanishes for $j\ne k-1$ and is well known to give a projective
$R$-module for $j=k-1$. A similar calculation with a local coefficient
cohomology Serre spectral sequence, together with the fact that
$\Stein_k$ is self-dual, completes the proof.
\end{proof}

\section{Examples}
\label{section: our coefficients}

In this section we describe some particular Mackey functors for
$\Sigma_n$, and show that they satisfy the hypotheses of
Theorem~\ref{theorem: main theorem} and
Corollary~\ref{corollary:only-one-dimension}. The general construction
is described in Definition~\ref{defn:our-mackey-functors} and
Proposition~\ref{proposition:gamma-f-is-mackey} below.

We will consider Mackey functors with values in categories other than
abelian groups. If $\Ccal$ is any additive category, then one defines a
Mackey functor with values in $\Ccal$ to be a pair of additive
functors $(\gamma, \gamma^\natural)$ from the category of finite
$G$-sets to $\Ccal$, satisfying the same hypotheses as in
Definition~\ref{definition: Mackey functor}. In particular, we will
consider Mackey functors with values in the category of graded abelian
groups (graded Mackey functors).

\begin{remark}  \label{remark: composition of Mackey functors}
  Note that if $M$ is a Mackey functor with values in $\Ccal$, and
  $F\colon\Ccal\to \Dcal$ is an additive functor between additive
  categories, then $F\circ M$ is a Mackey functor with values in
  $\Dcal$.
\end{remark}

Our basic construction of a Mackey functor involves the homotopy
category of spectra with an action of $G$. This is unsurprising, as
the connection between Mackey functors and equivariant stable homotopy
theory is well known. We will not require the full strength of this
theory, and in particular we will only use a naive version of
equivariant stable homotopy theory.  By the {\it homotopy category of
  spectra with an action of $G$}, we mean the category of spectra with
an action of $G$, localized with respect to equivariant maps that are
weak equivalences of the underlying nonequivariant
spectra. Equivalently, this is the stable homotopy category of
$G$-spaces, or $G$-simplicial sets. This category of spectra with an
action of $G$ supports a Quillen model structure, where fibrations and
weak equivalences are defined in the underlying category of
spectra. (See, for example~\cite{Schwede} or~\cite[Theorem
11.6.1]{Hirschhorn}.) Therefore, its homotopy category is
well-defined.  Note also that the homotopy category of spectra with an
action of $G$ is an additive category.

Define the functor $\gamma$ from the category of finite $G$-sets to
the homotopy category of spectra with an action of $G$ by the formula
$\gamma(S)=\Sigma^\infty S_+$. The following is a standard fact.

\begin{lemma}\label{lem:suspension-is-mackey}
  The functor $\gamma$ extends naturally to a Mackey functor with
  values in the homotopy category of spectra with an action of $G$.
\end{lemma}

\begin{proof}
  The extended functor $\gamma$ is additive by construction. The
  contravariant functor $\gamma^\natural$ is necessarily defined on
  objects to be the same as~$\gamma$. To define $\gamma^\natural$ on
  morphisms, let the superscript $^\vee$ denote the Spanier-Whitehead
  dual of a spectrum.  Given a finite set $S$,
  there is a weak equivalence 
  $\Sigma^\infty S_+ \to (\Sigma^\infty S_+)^\vee$. This equivalence
  is natural with respect to set isomorphisms;
  in particular, if $S$ is a $G$-set, the equivalence is
  $G$-equivariant. Let $S\to T$ be a $G$-map between finite $G$-sets.
  We define the map $\gamma^\natural(T)\to \gamma^\natural(S)$ as
  the composite
\[
\Sigma^\infty T_+ \to (\Sigma^\infty T_+)^\vee 
                  \to (\Sigma^\infty S_+)^\vee 
                  \xleftarrow{\,\simeq\,}\Sigma^\infty S_+.
\]
Note that the ``wrong way map'' is a weak equivalence, and therefore
is invertible in the homotopy category of spectra with an action of
$G$.

Consider the diagrams below. 
That $M=(\gamma, \gamma^\natural)$ is a Mackey functor follows from
the fact that given a pullback diagram on the left of
finite $G$-sets, the diagram on the right commutes in the homotopy
category of spectra with an action of~$G$:
  \[
    \begin{CD}
        S@>u>> T\\
       @V\alpha VV    @V\beta VV\\
        U @>v>> V
     \end{CD}
\qquad\qquad\qquad
     \begin{CD}
     \Sigma^\infty S_+        @>\gamma(u) >> \Sigma^\infty T_+\\
     @A{\gamma^\natural(\alpha)}AA   @A{\gamma^\natural(\beta)}AA     \\
        \Sigma^\infty U_+ @>\gamma(v)>> \Sigma^\infty V_+
     \end{CD}\,.
\]     
\nopagebreak
\end{proof}

We say that a functor from spectra to spectra is \defining{additive} if
it respects equivalences and preserves finite coproducts up to
equivalence. Recall that $\Sigma_n$ acts on the one-point
compactification $S^n$ of $\reals^n$ by permuting coordinates, and
hence on the $j$-fold smash product $S^{nj}$.  The following is the
general construction of Mackey functors that we wish to consider.

\begin{defn}\label{defn:our-mackey-functors}
  Suppose that $j$ is a fixed integer, with $j$ odd if $p$~is odd, and
  that $F$ is an additive functor from spectra to spectra.  For each
  finite $\Sigma_n$-set $T$, define the graded abelian group
  $\coeff_F(T)$ by
  \[
       \coeff_F(T)=\pi_*F\left((\Sigma^\infty T_+\wedge S^{nj})_{\hobased
         \Sigma_n}\right)\,.
  \]
\end{defn}

Our main result in this section is
Proposition~\ref{proposition:gamma-f-is-mackey} below. Let $L_{(p)}$
denote the functor on spectra given by localization at $p$.

\begin{proposition}\label{proposition:gamma-f-is-mackey}
  The assignment $T\mapsto \coeff_F(T)$ extends naturally to a Mackey
  functor for $\Sigma_n$ that satisfies the centralizer condition
  (see Definition~\ref{definition: Mackey functor condition}). If
  $F\to F\circ L_{(p)}$ is an equivalence, then $\coeff_F$ takes
  values in $\integers_{(p)}$-modules, is projective relative to
  $p$-subgroups, and (if $p$ is odd) satisfies the involution
  condition.
\end{proposition}

We give the proof of Proposition~\ref{proposition:gamma-f-is-mackey}
after a few examples. 

\begin{example}
  If $F(X)= H\field_p\wedge X$, then $\coeff_F(T)$ is the
 $\field_p$-homology of the relevant reduced Borel construction. 
  In this
  case Proposition~\ref{proposition:gamma-f-is-mackey} and
  Theorem~\ref{theorem: main theorem}, taken together, give a relatively
  conceptual approach to the homological calculations of
  Arone-Mahowald in \cite{A-M}. 
\end{example}

\begin{example}
  One advantage of our approach is that it applies in situations in
  which explicit calculation is impossible, e.g., when $F=L_{(p)}$
  itself.  Here $\coeff_F(T)$ is the $p$-local \emph{stable homotopy}
  of the given Borel construction.  The calculation of Bredon homology
  in this case provides a key ingredient for a new proof of some
  theorems of Kuhn~\cite{Kuhn-Whitehead} and
  Kuhn-Priddy~\cite{Kuhn-Priddy} on the Whitehead conjecture (see also~\cite{Kuhn-Whitehead-final} for Kuhn's latest word on the subject).  The
  Bredon cohomology also leads to a new proof of
  the
  collapse of the homotopy spectral sequence of the Goodwillie tower
  of the identity functor evaluated at $S^1$. This was done by  Behrens~\cite{Behrens-Goodwillie-Tower} at the prime $2$ and then by Kuhn~\cite{Kuhn-Whitehead-final} at all primes.  We intend to pursue
  this in another paper.
\end{example}

\begin{example}
Another interesting example to which our results apply is the functor 
\[
F(X)=(E\wedge X)_K.
\]
Here $E$ is the Morava $E$-theory and the subscript $_K$ denotes
localization with respect to Morava $K$ theory. This example, and
others similar to it, were considered recently by Rezk~\cite{Rezk} and
Behrens~\cite{Behrens-Rezk}. It seems that our methods can be used to
recover some of their calculations. For example, Lemma 5.6
of~\cite{Behrens-Rezk} seems to be closely related to our main
theorem, applied to the functor $F$ above.
\end{example}

\begin{proof}[Proof of Proposition~\ref{proposition:gamma-f-is-mackey}]
  We saw in Lemma~\ref{lem:suspension-is-mackey} that the functor
  $T\mapsto \Sigma^\infty T_+$ extends to a Mackey functor. The
  functor $\coeff_F$ is obtained by composing the suspension spectrum
  functor with the following functors: smash product with $S^{nj}$,
  taking $\Sigma_n$-homotopy orbits, $F$, and $\pi_*$. Each
  of these functors is additive (on the level of homotopy
  categories), and therefore $\coeff_F$ extends to a Mackey functor by
  Remark~\ref{remark: composition of Mackey functors}. 

  Next we claim that the Mackey functor
  $T\mapsto L_p(\Sigma^\infty T_+ \wedge S^{nj})_{\hobased \Sigma_n}$
  (with values in the homotopy category of spectra) has the
  $p$-transfer property. By this we mean that for every
  $\Sigma_{n}$-set $Z$ of cardinality prime to $p$, the following
  composed map is an equivalence, i.e., an isomorphism in the homotopy
  category of spectra:
\begin{align*}
 L_p\left(\Sigma^\infty T_+ \wedge S^{nj}\right)_{\hobased \Sigma_n} 
    &\longrightarrow
 L_p\left(\Sigma^\infty (Z\times T)_+ \wedge S^{nj}
           \right)_{\hobased \Sigma_n} \\
    &\longrightarrow
 L_p(\Sigma^\infty T_+ \wedge S^{nj})_{\hobased \Sigma_n}.
\end{align*}
(See Definition~\ref{definition: p-constrained}.)
To see this, note that our functor is 
equivalent to 
$T\mapsto \left(L_p(\Sigma^\infty T_+) \wedge
               S^{nj}\right)_{\hobased \Sigma_n}$, 
so it is enough to prove that $L_p(\Sigma^\infty T_+)$ has the 
$p$-transfer property in the same sense: 
\begin{equation}     \label{eq: p-equivalence}
 L_p\left(\Sigma^\infty T_{+}\right)
  \longrightarrow
 L_p\left(\Sigma^\infty \left(Z\times T\right)_{+}\right)
  \longrightarrow
 L_p\left(\Sigma^\infty T_{+}\right)
\end{equation}
should be an equivalence of $\Sigma_{n}$-spectra for every
$\Sigma_{n}$-set $Z$ of cardinality prime to~$p$.  However, the effect
of the composition
$\Sigma^\infty T_+ \to \Sigma^\infty (Z\times T)_+ \to \Sigma^\infty T_+$
on homology is multiplication by $|Z|$.  Thus it induces an
isomorphism on homology with coefficients in $\integers_{(p)}$, and
\eqref{eq: p-equivalence} is an equivalence.  Since
$F\simeq F\circ L_p$, it follows that $\coeff_F$ has the $p$-transfer
property. By~Lemma~\ref{lemma:p-transfer}, $\coeff_F$ takes values in
$\integers_{(p)}$-modules, and is projective relative to
$p$-subgroups.

Next we consider the centralizer condition. To set the stage, consider
a general situation where $Z$ is a space with a pointed action
of~$\Sigma_n$.  Let $D$ be a subgroup of~$\Sigma_n$, with
centralizer~$\Cen_{\Sigma_{n}}(D)$. The orbit space of the action of $D$
on $Z$ still has an action of $\Cen_{\Sigma_{n}}(D)$, and 
there is a homeomorphism between two 
models of this $\Cen_{\Sigma_{n}}(D)$-space given by: 
\begin{align*}
\phi: Z/D &\xrightarrow{\ \cong\ }
   \left(\Sigma_n/D_+ \wedge Z\right)/{\Sigma_n}\\
Dz&\ \mapsto(eD, z). 
\end{align*}
The action of $c\in\Cen_{\Sigma_{n}}(D)$
on the domain is given by $c(Dz)=D(cz)$ and on the codomain is 
given by $c(\sigma D, z)=(\sigma c^{-1} D, z)$, 
and $\phi$ is $\Cen_{\Sigma_{n}}(D)$-equivariant with respect to 
these actions. 

We will apply the paragraph above with
$Z=(E\Sigma_{n})_{+}\wedge S^{nj}$, where $\Cen_{\Sigma_{n}}(D)$
acts on $Z$ diagonally. 
Applying the discussion in the previous paragraph gives us a 
$\Cen_{\Sigma_{n}}(D)$-equivariant homeomorphism 
% $\left(
% \Sigma_n/D_+ \wedge Z
% \right)/\Sigma_{n}
% $
\begin{equation}   \label{eq: homeomorphim}
\phi: \left[(E\Sigma_{n})_{+}\wedge S^{nj}\right]/D
  \xrightarrow{\ \cong\ } 
\left[\Sigma_{n}/D_{+}\wedge (E\Sigma_{n})_{+}\wedge S^{nj}\right]/\Sigma_{n}.
\end{equation}
The right-hand side is
$(\Sigma_n/D_+ \wedge S^{nj})_{\hobased\Sigma_n}$, and the centralizer
condition will be satisfied if we can show that the kernel of
$\Cen_{\Sigma_{n}}(D) \rightarrow\pi_{0}\Cen_{\GL_n\reals}(D)$ acts on
this space via maps that are homotopic to the identity. 

We establish what is required by using the left-hand side of
\eqref{eq: homeomorphim} instead.  Suppose
$c\in\ker\left[ \Cen_{\Sigma_{n}}(D)\rightarrow
  \pi_{0}\Cen_{\GL_n\reals}(D)\right]$.
Then the action of $c$ on $S^{nj}$ is homotopic to the identity
through $D$-equivariant maps.  Likewise, translation by $c$ on
$E\Sigma_{n}$ is homotopic to the identity, and because $c$
centralizes~$D$, the homotopy is through $D$-equivariant maps.  It
follows that the action of $c$ on
$\left[(E\Sigma_{n})_{+}\wedge S^{nj}\right]/D$ is homotopic to the
identity, and the same is true of the action on
$(\Sigma_n/D_+ \wedge S^{nj})_{\hobased\Sigma_n}$.  We conclude that
the action of $c$ on $\coeff_F\left(\Sigma_{n}/D\right)$ is trivial, and
hence $\coeff_F$ satisfies the centralizer condition.

% is $(ED)_{+}\wedge_{D} S^{nj}$, with the diagonal action of
% $\Cen_{\Sigma_{n}}(D)$, so the centralizer condition will be satisfied
% if $\Cen_{\Sigma_{n}}(D)$ acts on both $ED_{+}$ and $S^{nj}$ by maps
% that are homotopic to the identity through $D$-equivariant maps.

%  \bigskip

%  Hence $c$ induces a
% self-map of $\left(S^{nj}\right)_{\hobased D}$ that is homotopic to
% the identity.  
% One can check that the action of $\Cen_{\Sigma_{n}}(D)$
% on $\coeff_F(G/D)$ referenced in Definition~\ref{definition:
%   centralizer condition} is the same as that induced by the action of
% $C$ on $\left(S^{nj}\right)_{\hobased D}$ from the action of
% $\Cen_{\Sigma_{n}}(D)$ on $S^{nj}$.
%

It remains to prove that if $F\circ L_{(p)}\simeq F$, and $p$ is odd,
then $\coeff_F$ satisfies the involution condition.  Suppose in
general that $\tau$ is an involution of a spectrum $X$ that has $2$
invertible in~$\pi_*(X)$. 
% (If $X$ has a group action, $\tau$ is
% required to act by an equivariant map.) 
Then $\pi_*(X)$ splits as a
direct sum of eigenspaces for $1$ and~$-1$. It follows that $\tau$
acts by $-1$ on $\pi_*(X)$ if and only if the map $\tau - 1$ is an
equivalence.  
% This characterization has the following consequence:
% suppose that $F$ is an additive functor, $X$ is a spectrum with $2$
% invertible, and $\tau$ is an involution of $X$ that acts by $-1$ on
% the homotopy groups of $X$. Then $\tau$ acts by $-1$ on the homotopy
% groups of $F(X)$.

For our situation, let $D$ and $\Cen_{\Sigma_{n}}(D)$ be as above, and
let $\tau$ be an odd involution in $\Cen_{\Sigma_{n}}(D)$. We need to
show that $\tau$ acts as multiplication by $-1$ on
$\pi_*F\left(\Sigma^\infty S^{nj}\right)_{\hobased D}$.  Since
$F\simeq F\circ L_p$, and $L_p$ commutes with homotopy colimits, it is
enough to prove that $\tau$ acts by $-1$ on
$ \pi_*F\left(L_p\left(\Sigma^\infty S^{nj}\right)_{\hobased
    D}\right)$.
However, the action of $\tau$ is induced by the
action of $\tau$ on $S^{nj}$. Since $\tau$ is an odd involution,
$\tau$ acts by a map of degree~$-1$ on $S^n$ and hence, since $j$ is
odd, by a map of degree~$-1$ on $S^{nj}$. It follows that $\tau$ acts
by $-1$ on $\pi_*L_p\left(\Sigma^\infty S^{nj}\right)$.  Since $p$ is
odd, $2$ is invertible in $\pi_*L_p\left(\Sigma^\infty S^{nj}\right)$.
% In the case at hand, since homotopy orbits and $F$ are both additive
% functors, it follows from the preceding paragraph that $\tau$ acts by
% $-1$ on $\pi_*F\left(L_p\left( \Sigma^\infty S^{nj}\right)_{\hobased D}\right)$, 
% which is what we had to prove.
By the previous paragraph, the map $\tau-1$ induces a
self-equivalence on $L_p(\Sigma^\infty S^{nj})$. 

Since $\tau$ is in
the centralizer of $D$, it acts on $L_p(\Sigma^\infty S^{nj})$ by a
$D$-equivariant map. Therefore, the map $\tau-1$ is
$D$-equivariant. It follows that $\tau-1$ induces a self-equivalence
of $L_p(\Sigma^\infty S^{nj})_{\hobased D}$. Since $F$ is an additive functor,
it follows that $\tau-1$ induces a self-equivalence of
$F(L_p(\Sigma^\infty S^{nj})_{\hobased D})$. Again, by the same reasoning as above,
it follows that $\tau$ acts by $-1$ on
$\pi_*F(L_p(\Sigma^\infty S^{nj})_{\hobased D})$, which is what we wanted to
prove.
\end{proof}

\end{document}